\documentclass[a4paper]{amsart}

\usepackage{amsmath,amssymb,amsthm,mathtools}
\usepackage{tikz}
\usepackage{float}
\usepackage{setspace}
\allowdisplaybreaks
\theoremstyle{plain}
\newtheorem{theorem}{Theorem}[section]
\newtheorem{proposition}[theorem]{Proposition}
\newtheorem{lemma}[theorem]{Lemma}
\newtheorem{corollary}[theorem]{Corollary}
\theoremstyle{definition}
\newtheorem{definition}[theorem]{Definition}
\newtheorem{example}[theorem]{Example}
\newtheorem{introproblem}{Problem}
\newtheorem{prob}[theorem]{Problem}

\theoremstyle{remark}
\newtheorem{remark}[theorem]{Remark}

\numberwithin{equation}{section}

\newcommand{\C}{\mathbb{C}}
\newcommand{\R}{\mathbb{R}}
\newcommand{\End}{\operatorname{End}}
\newcommand{\id}{\operatorname{id}}
\newcommand{\esize}{\operatorname{esize}}
\newcommand{\DelsCode}{\operatorname{Dels}_{\mathrm{code}}}
\newcommand{\DelsDes}{\operatorname{Dels}_{\mathrm{des}}}

\title{Quantizing Delsarte Theory}

\author[R.~Misawa]{Ryutaro Misawa}
\address[R.~Misawa]{Division of Mathematics\\
Graduate School of Information Sciences\\
Tohoku University\\
6-3-09 Aramaki-Aza-Aoba, Aoba-ku, Sendai 980-8579\\
Japan}
\email{misawa.ryutaro.q2@dc.tohoku.ac.jp}

\begin{document}

\begin{abstract}
We develop Delsarte theory for quantum association schemes, allowing both composition and the quantum Schur product to be noncommutative.
A MacWilliams identity and positivity of the inner and dual distributions yield conic upper bounds for codes and lower bounds for designs in terms of effective size, which extends subset cardinality.
These bounds recover classical linear programs for symmetric association schemes and semidefinite programs for homogeneous coherent configurations.
We also quantize Schurian schemes using finite-group representations and characterize quantum block designs in quantum Johnson schemes arising from irreducible stabilizer representations by averaging over subgroups fixing prescribed points individually.
In a quantum Johnson scheme on \(M_2(\mathbb C)^{\oplus10}\), we prove that the minimum effective size of a nonzero positive element satisfying the strength-two design condition is \(20\), attained by complementary projections outside the center of the underlying algebra.
Every such element in the center has effective size \(40\), while the conic lower bound is \(40/3\), strictly below the actual minimum.
\end{abstract}

\maketitle

\section{Introduction}\label{sec:introduction}

Delsarte theory describes codes and designs in association schemes through distributions, derives upper bounds for code size and lower bounds for design size, and studies objects attaining them~\cite{DelsarteLevenshtein1998,MartinTanaka2009}. In 1973, Delsarte formulated their conditions as the vanishing of prescribed distribution components and obtained linear programming bounds from nonnegativity~\cite{Delsarte1973}. We extend this central framework of algebraic combinatorics to finite quantum spaces.

Codes and designs arose independently from error correction and prescribed uniformity. Classical design theory developed through combinatorial constructions and experimental planning: Fisher studied treatment comparisons accounting for experimental conditions~\cite{Fisher1926}, Bose studied equal pairwise block incidences~\cite[Introduction]{Bose1939}, and orthogonal arrays reduced the number of factor-level combinations needed in experiments~\cite{PlackettBurman1946,Rao1947}. Shannon established the foundations of communication over noisy channels in 1948~\cite{Shannon1948}, and Hamming studied error-detecting and error-correcting codes in 1950~\cite{Hamming1950}.

In a block design, every point subset of a prescribed size lies in the same number of blocks of a fixed size; when this number is one, the design is a Steiner system~\cite{Stinson2004}. Subspace designs are finite-field \(q\)-analogs, replacing subsets and their sizes by subspaces and their dimensions~\cite{BraunKiermaierWassermann2018}. An orthogonal array of strength \(t\) requires each symbol combination to occur equally often in any \(t\) columns~\cite{HedayatSloaneStufken1999}; mutually orthogonal Latin squares give examples of strength two~\cite{Stinson2004}.

Bose and Shimamoto studied association schemes in 1952, and Bose and Mesner studied their adjacency algebras in 1959~\cite{BoseShimamoto1952,BoseMesner1959}. A scheme partitions the ordered pairs of a finite set into relations, including the diagonal, whose adjacency algebra is closed under matrix multiplication, transpose, and the entrywise Schur product. It is commutative when matrix multiplication is commutative and symmetric when every relation is invariant under reversing pairs; symmetry implies commutativity~\cite{MartinTanaka2009}. Delsarte theory places codes and orthogonal arrays in Hamming schemes, block designs in Johnson schemes, and subspace designs in Grassmann schemes, whose relations are determined by Hamming distance, intersection size, and intersection dimension, respectively~\cite{Delsarte1973,MartinTanaka2009,BraunKiermaierWassermann2018}.

The inner distribution records normalized pair counts, while in a symmetric scheme the dual distribution records normalized squared norms of eigenspace projections. Code and design conditions require prescribed inner and dual components, respectively, to vanish. Eigenmatrices transform these distributions; in the Hamming scheme, this includes the MacWilliams identity relating the weight distributions of a linear code and its dual~\cite{MacWilliams1963,Delsarte1973}.

Higman's coherent configurations allow the diagonal to split into several relations and the adjacency algebra to be noncommutative~\cite[Section~3]{Higman1970}; those with a single diagonal relation are homogeneous. Hobart derived semidefinite bounds on subset size~\cite{Hobart2009}, and Bamberg and Lansdown used designs in homogeneous coherent configurations to study transitive permutation groups~\cite{BambergLansdown2026}. These extensions still concern finite sets and have commutative Schur products.

Quantum information has independently motivated quantum codes and designs. Shor constructed a quantum error-correcting code in 1995, and Knill and Laflamme characterized exact error correction in 1997~\cite{Shor1995,KnillLaflamme1997}. Shor and Laflamme introduced quantum weight enumerators and a MacWilliams identity, which Rains used to bound code dimension~\cite{ShorLaflamme1997,Rains1998,Rains1999}.

Zauner defined quantum designs as finite families of orthogonal projections in finite-dimensional matrix algebras in his 1999 thesis, recovering classical structures from commuting projections~\cite[Definition~1.1 and Theorem~1.10]{Zauner1999}.
For a fixed rank strictly between zero and the matrix size, allowing arbitrary projections enlarges the finite set of diagonal projections to a continuous Grassmann manifold~\cite[Section~1.4]{Zauner1999}.
Goedicke and Vicary unified classical and quantum block designs using completely positive maps between finite-dimensional algebras~\cite[Sections~5.2--5.3]{GoedickeVicary2023}.
These definitions do not include a fixed ambient association scheme and its distributions.
Our block construction retains the finite classical block index, quantizes its Johnson scheme, and defines designs as positive elements in the resulting finite quantum space (Sections~\ref{sec:schurian}--\ref{sec:quantum-block-designs}).

Goyeneche and {\.Z}yczkowski related orthogonal arrays to quantum states with maximally mixed reductions~\cite{GoyenecheZyczkowski2014}, and Goyeneche et al.\ introduced quantum orthogonal arrays by replacing rows with states and imposing partial-trace conditions~\cite{GoyenecheRaissiDiMartinoZyczkowski2018}. Quantum Latin squares similarly replace symbols with vectors forming orthonormal bases in each row and column~\cite{MustoVicary2016}. We adopt the quantum orthogonal array definition of Du et al.~\cite{DuYinPangWang2020} and Zang et al.~\cite[Definition~4.1]{ZangFacchiTian2021}, which includes cross terms between rows: for standard basis rows, it is equivalent to a classical orthogonal array of strength \(t\) whose rows remain distinct after deleting any \(t\) columns.

Gromada introduced quantum association schemes (QAS) within quantum graph theory~\cite{Gromada2024}. They replace function algebras of finite sets with finite-dimensional \(C^\ast\)-algebras \(A\), equipped here with the trace \(h_A\) restricting to \(d\operatorname{Tr}\) on each summand \(M_d(\C)\). When \(A\) is commutative, it is the function algebra of a finite set. A QAS is an operator algebra on \(A\) closed under composition and the quantum Schur product extending entrywise multiplication (Definition~\ref{def:qas}); both products can be noncommutative. It is commutative or cocommutative when composition or the quantum Schur product, respectively, is commutative, and is a CCQAS when both are~\cite[Definition~2.1 and Section~2.3]{Gromada2024}. Classical schemes (AS) give cocommutative QAS; commutative classical schemes (Com AS), including symmetric schemes (Sym AS), give CCQAS (Figure~\ref{fig:qas-inclusions}).

\begin{figure}[htbp]
\centering
\begin{tikzpicture}[
  x=1cm,y=1cm,
  font=\small,
  every node/.style={align=center},
  line width=.7pt
]
  \fill[black!4] (4.7,.4) rectangle (7.5,2.55);
  \draw (0,0) rectangle (12.6,3.2);
  \node[anchor=west] at (.2,3.0) {QAS};
  \draw (.3,.25) rectangle (7.5,2.7);
  \draw (4.7,.4) rectangle (12.3,2.55);
  \node at (2.5,2.3) {Com QAS};
  \node[font=\footnotesize] at (2.5,1.9)
    {Commutative composition};
  \node at (9.9,2.3) {Cocom QAS};
  \node[font=\footnotesize] at (9.9,1.9)
    {Commutative Schur product};
  \node at (6.1,2.3) {CCQAS};
  \draw[dash pattern=on 4pt off 3pt]
    (5.0,.55) rectangle (12.0,1.4);
  \node[fill=white,inner sep=2pt] at (11.6,1.4) {AS};
  \node at (6.2,1.2) {Com AS};
  \node at (9.8,.95) {Non-Com AS};
  \draw (5.3,.65) rectangle (7.1,1.0);
  \node[font=\footnotesize] at (6.2,.825) {Sym AS};
\end{tikzpicture}
\caption{Inclusions among classes of QAS.}
\label{fig:qas-inclusions}
\end{figure}
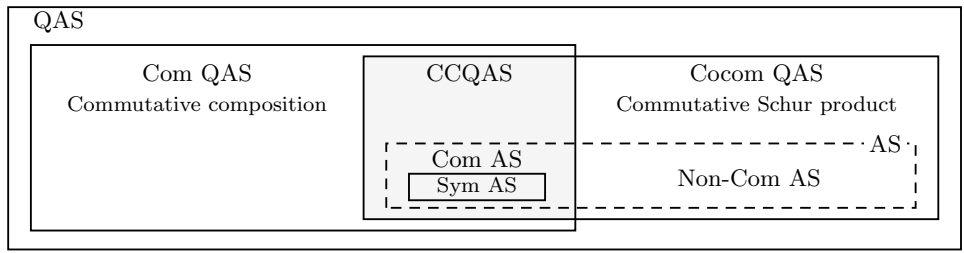

Okada obtained Delsarte-type linear programming bounds for quantum metric spaces when the isometry-group representation on the operator space is multiplicity-free~\cite[Chapter~3]{Okada2025}. Kubischta and Teixeira allowed multiplicities and derived matrix-valued MacWilliams identities and semidefinite bounds~\cite[Section~III]{KubischtaTeixeira2026}. Although they also define enumerators without a group action, their canonical transform and semidefinite program use one~\cite[Sections~II--III]{KubischtaTeixeira2026}. Our theory starts from an arbitrary QAS, requires no additional group action or metric, and uses both products to bound codes from above and designs from below.

We formulate five guiding problems.

Classical constructions include translation schemes, invariant under group translations, and Schurian schemes, defined by ordered-pair orbits of transitive group actions~\cite[Section~3]{Higman1970}.

\begin{introproblem}\label{prob:intro-quantization}
Quantize classical association schemes on noncommutative finite quantum spaces, recovering them as classical special cases. Construct QAS admitting no isomorphism to a classical adjacency algebra that preserves both products.
\end{introproblem}

Gromada quantized abelian translation schemes~\cite[Section~4.3]{Gromada2024}; we quantize Schurian schemes. Further examples requiring other constructions include the distance schemes of twisted Grassmann graphs of diameter at least two, which are neither translation nor Schurian because the graphs are not vertex-transitive~\cite{VanDamKoolen2005}.

\begin{introproblem}\label{prob:intro-design-quantization}
Quantize combinatorial designs, including block designs and orthogonal arrays, with an explicit classical correspondence. Determine existence and nonexistence conditions and give explicit constructions.
\end{introproblem}

Existing quantum designs recover classical objects in different ways. A common framework should express their uniformity on finite quantum spaces and include further examples such as subspace designs.

\begin{introproblem}\label{prob:intro-theory}
Develop Delsarte theory for general QAS. Determine which classical results extend and how quantum codes and designs correspond to QAS code and design conditions.
\end{introproblem}

This is related to Problem~(vii) in Zauner's epilogue, which proposes quantum designs with association schemes as a generalization of partially balanced incomplete block designs~\cite[p.~69]{Zauner1999}. The aim is to preserve the classical code--design correspondence under quantization.

A subset's characteristic function is a projection in its function algebra. We impose code and design conditions more generally on nonzero positive elements \(\chi\in A\), with effective size
\[
\operatorname{esize}(\chi):=\frac{h_A(\chi)^2}{h_A(\chi^2)}.
\]
For a nonzero projection \(p\), this equals \(h_A(p)\), and for a subset's characteristic function it equals cardinality. Central projections select whole matrix summands; a projection outside the center of \(A\) selects a proper nonzero subspace in at least one summand.

\begin{introproblem}\label{prob:intro-construction}
For a fixed QAS and prescribed code or design conditions, determine and attain the maximum effective size of codes and the minimum effective size of designs. When do projections outside the center of \(A\) improve on central projections?
\end{introproblem}

Feasible distributions need not come from positive elements, so attaining bounds also requires studying realizability.

\begin{introproblem}\label{prob:intro-bounds}
Classify codes and designs attaining the bounds. Determine when distributions come from positive elements and when codes or designs give rise to a QAS.
\end{introproblem}

\subsection{Contributions}

These problems remain open in general; our contributions are as follows.
\begin{enumerate}
\item We prove a MacWilliams identity and conic upper bounds for codes and lower bounds for designs in general QAS. These bounds recover linear programs for symmetric schemes and semidefinite programs for homogeneous coherent configurations (Theorems~\ref{thm:macwilliams} and~\ref{thm:conic-bounds}, Proposition~\ref{prop:classical-noncommutative-specialization}, and Corollary~\ref{cor:classical-programs}).

\item We quantize Schurian schemes using finite-group representations. For irreducible \(\tau\), averaging over subgroups fixing prescribed points individually characterizes quantum block \(\tau\)-designs as designs in the quantum Johnson scheme, with trivial \(\tau\) recovering weighted classical block designs (Theorems~\ref{thm:quantum-schurian-construction} and~\ref{thm:johnson-pointwise-characterization}).

\item We characterize diameter-two Moore graphs~\cite{HoffmanSingleton1960} by code and design conditions, and pure quantum codes with a common prime-power subsystem dimension by QAS code conditions. Under the same dimension assumption, quantum orthogonal arrays are characterized by a normalized row sum and a rank-one design condition in the Weyl realization of the Hamming scheme (Propositions~\ref{prop:moore-correspondence} and~\ref{prop:weyl-code-distributions}, and Theorem~\ref{thm:qoa-qas-design}).

\item In a quantum Johnson scheme on ten copies of \(M_2(\C)\), the minimum effective size of a nonzero positive element satisfying the strength-two design condition is \(20\), attained by complementary projections outside the center of \(A\). Every such element in the center has effective size \(40\) (Theorem~\ref{thm:johnson-minimum-size}).

\item We compute exact conic bounds for a family of QAS on matrix algebras with both products noncommutative. In the Johnson example, replacing only inner-distribution positive semidefiniteness with trace nonnegativity gives a lower bound strictly below the full conic bound, which is strictly below the minimum effective size: \(10<40/3<20\) (Proposition~\ref{prop:matrix-difference-bounds} and Theorem~\ref{thm:johnson-conic-values}).
\end{enumerate}

Sections~\ref{sec:classical-delsarte}--\ref{sec:delsarte-theory} develop the theory, Sections~\ref{sec:moore}--\ref{sec:quantum-examples} treat examples, Sections~\ref{sec:schurian}--\ref{sec:quantum-block-designs} quantize Schurian schemes and block designs, and Section~\ref{sec:optimal-block-design} determines the optimal designs and conic bounds in the Johnson example.
\section{Classical Delsarte theory}\label{sec:classical-delsarte}

The inner and dual distributions describe a subset by its pair counts and projections onto common eigenspaces. Their positivity and change of basis give the linear programming bounds extended to QAS in Section~\ref{sec:delsarte-theory}.

For \(r\in\mathbb Z_{\ge0}\), put \([r]:=\{1,\ldots,r\}\).
For a nonempty finite set \(X\), write \(\binom{X}{r}:=\{B\subseteq X:|B|=r\}\), and let \(\C^X\) be the space of complex functions on \(X\), represented as column vectors.
Write \(I\) for the identity matrix and \(J\) for the all-ones matrix.
We use \(\dagger\) for conjugate transpose and \(\succeq0\) for positive semidefiniteness.

\begin{definition}\label{def:symmetric-as}
For \(d\in\mathbb Z_{\ge0}\), a symmetric association scheme of class \(d\) on \(X\) consists of nonzero \(0,1\)-matrices \(B_0,\ldots,B_d\) indexed by \(X\) and satisfying
\[
\begin{gathered}
B_0=I,\qquad
\sum_{i=0}^dB_i=J,\qquad
B_i^T=B_i,\\
B_iB_j\in\operatorname{span}_{\C}\{B_0,\ldots,B_d\}
\quad(0\leq i,j\leq d).
\end{gathered}
\]
\end{definition}

The condition \(\sum_iB_i=J\) partitions the ordered pairs of \(X\) into relations. Their adjacency matrices span the Bose--Mesner algebra, a commutative \(C^\ast\)-algebra.
Put \(v:=|X|\).
Its primitive orthogonal idempotents \(E_0,\ldots,E_d\) satisfy
\[
E_iE_j=\delta_{ij}E_i,\qquad
E_j^\dagger=E_j,\qquad
\sum_{j=0}^dE_j=I,
\]
and are indexed so that \(E_0=v^{-1}J\)~\cite[Section~2]{MartinTanaka2009}.
For the entrywise Schur product \(\bullet\),
\[
B_i\bullet B_j=\delta_{ij}B_i.
\]
The adjacency matrices \(B_i\) describe pair relations, whereas the idempotents \(E_j\) project onto the common eigenspaces. They are bases for the same algebra, associated with its two products. Their change of basis defines the first eigenmatrix \(P=(P_{ji})\) and second eigenmatrix \(Q=(Q_{ij})\):
\begin{equation}\label{eq:classical-eigenmatrices}
B_i=\sum_{j=0}^dP_{ji}E_j,
\qquad
E_j=\frac1v\sum_{i=0}^dQ_{ij}B_i.
\end{equation}
In particular, \(P^{-1}=v^{-1}Q\).

\begin{definition}\label{def:classical-distributions}
For \(\varnothing\ne C\subseteq X\), let \(\chi_C\in\C^X\) be its characteristic vector. Its inner distribution \(a=(a_i)_{i=0}^d\) and dual distribution \(b=(b_j)_{j=0}^d\) are
\[
a_i:=\frac1{|C|}\chi_C^\dagger B_i\chi_C,
\qquad
b_j:=\frac v{|C|}\chi_C^\dagger E_j\chi_C
\qquad(0\le i,j\le d).
\]
\end{definition}

The entry \(a_i\) counts ordered pairs in \(C\times C\) belonging to the relation with adjacency matrix \(B_i\), divided by \(|C|\). The entry \(b_j\) is \(v/|C|\) times the squared norm of \(E_j\chi_C\).
The dual distribution here is \(|C|\) times that in~\cite[Theorem~3.11]{Delsarte1973}; the vanishing conditions are unchanged.

\begin{proposition}\label{prop:classical-distributions}
For every nonempty \(C\subseteq X\), its distributions satisfy
\[
\begin{aligned}
a_0&=1,& \sum_{i=0}^da_i&=|C|,\\
b_0&=|C|,& \sum_{j=0}^db_j&=v,
\end{aligned}
\]
the MacWilliams identity and its inverse,
\[
b_j=\sum_{i=0}^dQ_{ij}a_i,\qquad
a_i=\frac1v\sum_{j=0}^dP_{ji}b_j,
\]
and the nonnegativity conditions
\[
a_i\geq0\quad(0\leq i\leq d),\qquad
b_j\geq0\quad(0\leq j\leq d).
\]
\end{proposition}

\begin{proof}
The characteristic vector satisfies \(\chi_C^\dagger\chi_C=|C|\) and \(\chi_C^\dagger J\chi_C=|C|^2\). Applying \(B_0=I\), \(\sum_iB_i=J\), \(E_0=v^{-1}J\), and \(\sum_jE_j=I\) in the distribution formulas gives the four normalizations.
Substituting \eqref{eq:classical-eigenmatrices} into the distribution formulas gives the MacWilliams identity and its inverse~\cite[Section~2]{MartinTanaka2009}.
The inner distribution is nonnegative because the \(B_i\) have nonnegative entries; the dual distribution is nonnegative because \(E_j\succeq0\).
\end{proof}

Code conditions restrict pair relations, whereas design conditions require orthogonality to selected eigenspaces.

\begin{definition}\label{def:classical-code-design}
Let
\[
0\in\Delta\subseteq\{0,\ldots,d\},
\qquad T\subseteq\{1,\ldots,d\}.
\]
A nonempty subset \(C\subseteq X\) is a \(\Delta\)-code or a \(T\)-design if, respectively,
\[
a_i=0\quad(i\notin\Delta),
\qquad
b_j=0\quad(j\in T).
\]
\end{definition}

By the distribution formulas, \(a_i=0\) means that no ordered pair in \(C\times C\) belongs to the \(i\)-th relation, and \(b_j=0\) is equivalent to \(E_j\chi_C=0\).

To bound \(|C|\), we optimize over all real vectors satisfying these necessary conditions. Every subset gives a feasible vector, although the converse need not hold.
For \(0\in\Delta\subseteq\{0,\ldots,d\}\), the code linear program in \(a\in\R^{d+1}\) is
\begin{equation}\label{def:classical-code-program}
\begin{array}{ll}
\text{maximize} & \displaystyle\sum_{i=0}^d a_i,\\
\text{subject to} & a_0=1,\\
& a_i=0\quad(i\notin\Delta),\\
& a_i\geq0\quad(0\leq i\leq d),\\
& \displaystyle\sum_{i=0}^dQ_{ij}a_i\geq0\quad(0\leq j\leq d).
\end{array}
\end{equation}
For \(T\subseteq\{1,\ldots,d\}\), the design linear program in \(b\in\R^{d+1}\) is
\begin{equation}\label{def:classical-design-program}
\begin{array}{ll}
\text{minimize} & b_0,\\
\text{subject to} & b_j=0\quad(j\in T),\\
& b_j\geq0\quad(0\leq j\leq d),\\
& \displaystyle\frac1v\sum_{j=0}^dP_{ji}b_j\geq0\quad(0\leq i\leq d),\\
& \displaystyle\frac1v\sum_{j=0}^db_j=1.
\end{array}
\end{equation}

By Proposition~\ref{prop:classical-distributions}, the inner distribution of a \(\Delta\)-code is feasible for \eqref{def:classical-code-program}, and the dual distribution of a \(T\)-design is feasible for \eqref{def:classical-design-program}. Both have objective value \(|C|\), so these programs give an upper bound on code size and a lower bound on design size, respectively~\cite{Delsarte1973,MartinTanaka2009}.

The design program~\eqref{def:classical-design-program} uses a different normalization from the usual formulation of Delsarte's linear programming bound for designs~\cite[Sections~3.2 and~3.4]{Delsarte1973}; see also~\cite[Section~7]{MartinTanaka2009}.
For every feasible vector \(b\), its inverse transform \(a_i=v^{-1}\sum_jP_{ji}b_j\) satisfies \(a_i\ge0\), \(a_0=1\), and \(b_0=\sum_i a_i\ge1\).
Thus \(u_j:=b_j/b_0\) is well defined and satisfies \(u_0=1\), with the remaining homogeneous inequalities and vanishing conditions unchanged.
The inverse transformation is
\[
b_j=\frac{vu_j}{\sum_{\ell=0}^d u_\ell},\qquad
b_0=\frac{v}{\sum_{\ell=0}^d u_\ell}.
\]
Consequently, minimizing \(b_0\) is equivalent to maximizing \(\sum_j u_j\), with the minimum equal to \(v\) divided by the latter maximum; the two formulations therefore give the same design bound.

\section{Delsarte theory for quantum association schemes}\label{sec:delsarte-theory}

The classical bounds in Section~\ref{sec:classical-delsarte} use two nonnegative distributions related by the eigenmatrices.
For QAS, the simple ideals for either product can be matrix algebras.
We replace scalar distribution entries by linear functionals on these ideals and derive conic bounds for effective size.

\subsection{Finite quantum spaces and QAS}\label{sec:finite-quantum-spaces}

Finite quantum spaces replace function algebras on finite sets by finite-dimensional \(C^\ast\)-algebras.
We use the normalization of Gromada~\cite{Gromada2022Examples,Gromada2024}, which gives the usual sum over \(X\) when \(A=\C^X\).

\begin{definition}
A finite quantum space is a pair \((A,h_A)\), where \(A\ne\{0\}\) is a finite-dimensional unital \(C^\ast\)-algebra and
\[
A\cong\bigoplus_{\alpha=1}^rM_{n_\alpha}(\C),
\qquad
h_A((x_\alpha)_{\alpha=1}^r)
:=\sum_{\alpha=1}^rn_\alpha\operatorname{Tr}(x_\alpha).
\]
Here \(r,n_\alpha\in\mathbb Z_{\ge1}\), and \(\operatorname{Tr}\) is the usual matrix trace.
\end{definition}

The trace is independent of the chosen \(\ast\)-isomorphism~\cite[Examples~1.3--1.4 and Definition~1.5]{Gromada2024}.
We regard \(A\) as a Hilbert space with
\[
\langle x,y\rangle:=h_A(x^\ast y)
\qquad(x,y\in A),
\]
and use \(\dagger\) for adjoints with respect to this inner product and its tensor products.
The multiplication and unit maps are
\[
\begin{aligned}
m:A\otimes A&\longrightarrow A,& m(x\otimes y)&=xy,\\
\eta:\C&\longrightarrow A,& \eta(\lambda)&=\lambda1_A.
\end{aligned}
\]
With this normalization~\cite[Definitions~1.2 and~1.5]{Gromada2024},
\[
\eta^\dagger=h_A,\qquad
mm^\dagger=\id_A,\qquad
h_A(ab)=h_A(ba)\quad(a,b\in A).
\]
Write
\begin{equation}\label{eq:normalization-constant}
N:=(\eta^\dagger\eta)(1)=h_A(1_A)=\dim_\C A.
\end{equation}

Write \(\End(A):=\End_{\C}(A)\) and retain composition on this space.
The quantum Schur product and its involution are
\[
S\mathbin{\bullet}T:=m(S\otimes T)m^\dagger,
\qquad
S^\ast(x):=(S(x^\ast))^\ast
\qquad(S,T\in\End(A),\ x\in A).
\]
Set
\[
I:=\id_A,\qquad J:=\eta\eta^\dagger,
\qquad J(x)=h_A(x)1_A\quad(x\in A).
\]
Thus composition has unit \(I\) and involution \(\dagger\), whereas \((\End(A),\bullet,\ast)\) is a finite-dimensional \(C^\ast\)-algebra with unit \(J\)~\cite[Proposition~1.17]{Gromada2022Examples}.
The identities \(mm^\dagger=\id_A\) and \(J(x)=h_A(x)1_A\) give
\[
I\bullet I=I,\qquad J^2=NJ.
\]
Since \(\|1_A\|^2=N\),
\[
E_0:=N^{-1}J
\]
is the orthogonal projection onto \(\C1_A\).
For either algebra structure, projections are self-adjoint idempotents for its product and involution.

\begin{definition}\label{def:qas}
A quantum coherent algebra on \(A\) is a complex linear subspace \(W\subseteq\End(A)\) such that
\[
I,J\in W,\qquad
ST,\ S\bullet T,\ S^\dagger,\ S^\ast\in W
\qquad(S,T\in W).
\]
It is a quantum association scheme (QAS) if
\[
\left\{
p\in W:
p\bullet p=p=p^\ast,\ 
p\bullet I=I\bullet p=p
\right\}=\{0,I\}.
\]
A QAS is commutative or cocommutative, respectively, if
\[
ST=TS\quad(S,T\in W),
\qquad\text{or}\qquad
S\bullet T=T\bullet S\quad(S,T\in W).
\]
\end{definition}

The QAS condition says that \(I\) is a minimal nonzero projection in \((W,\bullet,\ast)\).
It is the homogeneity condition in~\cite[Definition~2.1]{Gromada2024}.

For \(A=\C^X\), the quantum Schur product is the entrywise product of operator matrices, and \(\ast\) is entrywise complex conjugation.
Thus the Bose--Mesner algebra of an association scheme on \(X\) is a cocommutative QAS, which is also commutative when the scheme is symmetric~\cite[Theorem~2.9]{Gromada2024}.

\subsection{Distributions and the MacWilliams identity}\label{sec:simple-ideals}\label{sec:distributions}

Fix a QAS \(W\subseteq\End(A)\), and write \(\circ\) for composition.
The structure theorem for finite-dimensional \(C^\ast\)-algebras~\cite[Theorems~III.1.1--III.1.2]{Davidson1996} gives
\[
\begin{aligned}
(W,\bullet,\ast)&\cong\bigoplus_{a\in\mathcal C}(V_a,\bullet,\ast),\\
(W,\circ,\dagger)&\cong\bigoplus_{i\in\mathcal S}(U_i,\circ,\dagger),
\end{aligned}
\]
where \(\mathcal C,\mathcal S\) index the simple ideals for the respective products.
Their identity elements \(z_a,e_i\) are the corresponding minimal central projections:
\[
V_a=z_a\bullet W,\qquad U_i=e_i\circ W,\qquad
\sum_{a\in\mathcal C}z_a=J,\qquad
\sum_{i\in\mathcal S}e_i=I.
\]

Homogeneity makes \(I\) a minimal projection in \((W,\bullet,\ast)\), while the rank-one operator \(E_0\) is a minimal projection in \((W,\circ,\dagger)\).
Choose the indices so that
\[
I\in V_0,\qquad E_0\in U_0.
\]
Here \(e_0\) is the identity of the whole ideal \(U_0\), while \(E_0\) projects onto \(\C1_A\); these projections need not coincide.

Projecting from one decomposition to the other gives the component maps
\[
\begin{aligned}
P_{ia}:V_a&\longrightarrow U_i,&
P_{ia}(x)&:=e_i\circ x,\\
Q_{ai}:U_i&\longrightarrow V_a,&
Q_{ai}(u)&:=Nz_a\bullet u,
\end{aligned}
\]
for \(a\in\mathcal C\) and \(i\in\mathcal S\).
The induced maps on the direct sums are
\[
P=(P_{ia}):\bigoplus_{a\in\mathcal C}V_a\longrightarrow
\bigoplus_{i\in\mathcal S}U_i,\qquad
Q=(Q_{ai}):\bigoplus_{i\in\mathcal S}U_i\longrightarrow
\bigoplus_{a\in\mathcal C}V_a.
\]

Summing the components identifies \(P\) with \(\id_W\) and \(Q\) with \(N\id_W\), since \(\sum_i e_i=I\) and \(\sum_a z_a=J\).
Consequently,
\begin{equation}\label{eq:inverse-block-maps}
PQ=N\id,\qquad QP=N\id
\end{equation}
on the respective direct sums, so \(P\) and \(N^{-1}Q\) are inverse.

For a symmetric association scheme, the \(B_a\) and \(E_i\) are the minimal projections for the respective products, so
\[
V_a=\C B_a,\qquad U_i=\C E_i
\quad(0\leq a,i\leq d).
\]
Since \(N=|X|\), the expansions in \eqref{eq:classical-eigenmatrices} identify \(P\) and \(Q\), in these bases, with the first and second eigenmatrices.

We now define the distributions of positive elements.
Write
\[
A_+:=\{y^\ast y:y\in A\}.
\]
For \(A=\C^X\), these are nonnegative functions, and projections are characteristic functions.
Their cardinality extends to the following quantity, which is unchanged by positive rescaling.

\begin{definition}\label{def:effective-size}
For \(\chi\in A_+\setminus\{0\}\), its effective size is
\[
\esize(\chi):=\frac{h_A(\chi)^2}{h_A(\chi^2)}.
\]
\end{definition}

When \(A=M_d(\C)\), the effective size is \(d\) times the quantity \((\operatorname{Tr}\chi)^2/\operatorname{Tr}(\chi^2)\) in~\cite[Eq.~(48)]{IosueEtAl2024}.

Classical distribution entries evaluate one basis operator in each ideal.
For QAS, we retain the quadratic evaluations of all operators in each ideal.

\begin{definition}\label{def:inner-distribution}\label{def:dual-distribution}
For \(\chi\in A_+\setminus\{0\}\), its inner distribution \(A^\chi=(A_a^\chi)_{a\in\mathcal C}\) and dual distribution \(B^\chi=(B_i^\chi)_{i\in\mathcal S}\) are defined by
\begin{equation}\label{eq:inner-distribution}
A_a^\chi:V_a\longrightarrow\C,\qquad
A_a^\chi(x):=\frac{\langle\chi,x(\chi)\rangle}{h_A(\chi^2)}
\qquad(x\in V_a)
\end{equation}
and
\begin{equation}\label{eq:dual-distribution}
B_i^\chi:U_i\longrightarrow\C,\qquad
B_i^\chi(u):=\frac{N\langle\chi,u(\chi)\rangle}{h_A(\chi^2)}
\qquad(u\in U_i).
\end{equation}
\end{definition}

These maps are complex linear.
For a family \(\mathsf A=(\mathsf A_a)_{a\in\mathcal C}\) of complex linear functionals \(\mathsf A_a:V_a\to\C\), write
\[
\mathsf A(w):=\sum_{a\in\mathcal C}\mathsf A_a(z_a\bullet w)
\qquad(w\in W).
\]
In particular,
\[
A^\chi(w)=\frac{\langle\chi,w(\chi)\rangle}{h_A(\chi^2)}.
\]

The following identities give the normalization and objective values used in the conic programs.

\begin{proposition}\label{lem:distribution-normalization}
Multiplying \(\chi\) by a positive scalar leaves its effective size and distributions unchanged.
The distributions satisfy
\[
\begin{aligned}
A_0^\chi(I)&=1,&
\sum_{a\in\mathcal C}A_a^\chi(z_a)&=A^\chi(J)=\esize(\chi),\\
B_0^\chi(E_0)&=\esize(\chi),&
\sum_{i\in\mathcal S}B_i^\chi(e_i)&=N,
\end{aligned}
\]
and
\[
A_a^\chi(x^\ast)=\overline{A_a^\chi(x)}
\quad(x\in V_a),\qquad
B_i^\chi(u^\dagger)=\overline{B_i^\chi(u)}
\quad(u\in U_i).
\]
Every nonzero projection \(p\in A\) has effective size \(\esize(p)=h_A(p)\).
\end{proposition}

\begin{proof}
Since \(\chi\ne0\), the denominator \(h_A(\chi^2)=\langle\chi,\chi\rangle\) is positive.
The numerators and denominator all scale by the square of the positive scalar, proving invariance under rescaling.
The normalization at \(I\) follows from \(I(\chi)=\chi\), while \(J(\chi)=h_A(\chi)1_A\) and \(E_0=N^{-1}J\) give the two values equal to \(\esize(\chi)\).
The sums follow by linearity from \(\sum_a z_a=J\) and \(\sum_i e_i=I\).
For the involution \(\ast\), traciality and \(\chi=\chi^\ast\) give
\[
\overline{\langle\chi,x(\chi)\rangle}
=h_A\bigl(x(\chi)^\ast\chi\bigr)
=h_A\bigl(\chi x(\chi)^\ast\bigr)
=\langle\chi,x^\ast(\chi)\rangle.
\]
The identity for \(B_i^\chi\) follows from the defining property of the Hilbert-space adjoint \(\dagger\).
Finally, \(p^2=p\) gives \(h_A(p^2)=h_A(p)\), proving the projection formula.
\end{proof}

For the characteristic function \(\chi_C\) of a nonempty subset in a symmetric association scheme, these definitions recover the distributions of Section~\ref{sec:classical-delsarte}:
\[
h_A(\chi_C^2)=|C|,\qquad
A_i^{\chi_C}(B_i)=a_i,\qquad
B_j^{\chi_C}(E_j)=b_j.
\]

The two distributions restrict the same functional on \(W\), with a factor of \(N\) for the dual distribution.
A prime denotes the complex linear dual, without conjugation.
The dual maps of \(P\) and \(Q\) act by
\[
(Q'\mathsf A)_i(u)=\sum_{a\in\mathcal C}\mathsf A_a(Q_{ai}(u)),
\qquad
(P'\mathsf B)_a(x)=\sum_{i\in\mathcal S}\mathsf B_i(P_{ia}(x)).
\]
Thus \((Q'\mathsf A)_i(u)=N\mathsf A(u)\) for \(u\in U_i\).

\begin{theorem}\label{thm:macwilliams}
The maps \(Q'\) and \(N^{-1}P'\) are inverse complex linear maps.
For every \(\chi\in A_+\setminus\{0\}\),
\begin{equation}\label{eq:block-macwilliams}
B^\chi=Q'A^\chi,\qquad
A^\chi=N^{-1}P'B^\chi.
\end{equation}
\end{theorem}

\begin{proof}
Dualizing \eqref{eq:inverse-block-maps} gives \(Q'P'=N\id\) and \(P'Q'=N\id\).
To identify the dual distribution, let \(u\in U_i\).
Since \(Q_{ai}(u)=Nz_a\bullet u\) and \(\sum_a z_a=J\),
\[
(Q'A^\chi)_i(u)=NA^\chi(u)
=\frac{N\langle\chi,u(\chi)\rangle}{h_A(\chi^2)}
=B_i^\chi(u).
\]
Applying the inverse map \(N^{-1}P'\) gives the second identity.
\end{proof}

For a symmetric association scheme, evaluating \eqref{eq:block-macwilliams} at the basis elements \(B_a\) and \(E_i\) recovers the scalar transforms in Proposition~\ref{prop:classical-distributions}.

\subsection{Codes, designs, and bounds}\label{sec:conic-bounds}

As in the classical case, codes restrict the support of the inner distribution, and designs require specified entries of the dual distribution to vanish.
Since an entry is now a functional, the condition applies to all elements of the corresponding ideal.

\begin{definition}\label{def:qas-code}\label{def:qas-design}
Let \(0\in\Delta\subseteq\mathcal C\) and \(T\subseteq\mathcal S\setminus\{0\}\).
An element \(\chi\in A_+\setminus\{0\}\) is a \(\Delta\)-code if
\[
A_a^\chi=0\qquad(a\notin\Delta),
\]
and a \(T\)-design if
\[
B_i^\chi=0\qquad(i\in T).
\]
\end{definition}

The following lemma checks the design condition directly on \(\chi\).

\begin{lemma}\label{lem:spectral-vanishing}
For \(\chi\in A_+\setminus\{0\}\) and \(i\in\mathcal S\),
\[
B_i^\chi=0\quad\Longleftrightarrow\quad e_i(\chi)=0.
\]
\end{lemma}

\begin{proof}
If \(e_i(\chi)=0\), then every \(u\in U_i\) annihilates \(\chi\), since \(u=u\circ e_i\).
This proves \(B_i^\chi=0\).
For the converse, it suffices to evaluate the functional at the orthogonal projection \(e_i\in U_i\):
\[
B_i^\chi(e_i)=\frac{N}{h_A(\chi^2)}\langle\chi,e_i(\chi)\rangle
=\frac{N}{h_A(\chi^2)}\|e_i(\chi)\|^2.
\]
Thus \(B_i^\chi=0\) implies \(e_i(\chi)=0\).
\end{proof}

To obtain bounds, we retain the linear and positivity conditions of distributions.
For the inner distribution, we test quadratic forms on \(A_+\); for the dual distribution, on all of \(A\).
Set
\[
\begin{aligned}
R_a^+&:=\left\{x\in V_a:
x^\ast=x,\ \langle\chi,x(\chi)\rangle\geq0\quad(\chi\in A_+)
\right\},\\
K_i^+&:=\left\{u\in U_i:
u^\dagger=u,\ \langle\xi,u(\xi)\rangle\geq0\quad(\xi\in A)
\right\}
\qquad(a\in\mathcal C,\ i\in\mathcal S).
\end{aligned}
\]

The cone \(K_i^+\) consists of positive operators on the Hilbert space \(A\), whereas \(R_a^+\) requires nonnegativity only on positive elements of the algebra \(A\).
For \(A=\C^X\), the latter tests nonnegative vectors, as in copositive programming~\cite{Duer2010}.

The variables of the conic programs are families of complex linear functionals \(\mathsf A_a\in V_a'\) and \(\mathsf B_i\in U_i'\) subject to
\begin{equation}\label{eq:conic-feasibility}
\begin{gathered}
\mathsf B=Q'\mathsf A,\qquad \mathsf A_0(I)=1,\\
\mathsf A_a(x^\ast)=\overline{\mathsf A_a(x)}\quad(x\in V_a),\qquad
\mathsf A_a(x)\geq0\quad(x\in R_a^+),\\
\mathsf B_i(u^\dagger)=\overline{\mathsf B_i(u)}\quad(u\in U_i),\qquad
\mathsf B_i(u)\geq0\quad(u\in K_i^+)
\end{gathered}
\end{equation}
for all \(a\in\mathcal C\) and \(i\in\mathcal S\).
These conditions retain the MacWilliams relation, normalization, compatibility with the involutions, and nonnegativity of actual distributions.
By Theorem~\ref{thm:macwilliams}, \(\mathsf B=Q'\mathsf A\) is equivalent to \(\mathsf A=N^{-1}P'\mathsf B\).
Since \(I\in V_0\), normalization also takes the form
\[
\frac1N\sum_{i\in\mathcal S}\mathsf B_i(e_i)=1,
\]
because the left-hand side equals \(\mathsf A(I)=\mathsf A_0(I)\).

\begin{proposition}\label{prop:distribution-positivity}
The sets \(R_a^+\) and \(K_i^+\) are closed convex cones in the respective self-adjoint real subspaces, and \(K_i^+\) consists of the positive operators in \(U_i\).
For every \(\chi\in A_+\setminus\{0\}\),
\[
A_a^\chi(x)\geq0\quad(x\in R_a^+),\qquad
B_i^\chi(u)\geq0\quad(u\in K_i^+).
\]
Every pair satisfying \eqref{eq:conic-feasibility} has equal nonnegative objectives,
\[
\mathsf B_0(E_0)=N\mathsf A(E_0)=\mathsf A(J)\geq0,
\]
and satisfies \(\mathsf A(w)\geq0\) for every \(w\in W\) that is positive as an operator on \(A\).
\end{proposition}

\begin{proof}
In their self-adjoint real subspaces, the cones are intersections of closed homogeneous half-spaces.
The quadratic evaluations defining \(R_a^+\) are real by traciality, as in the proof of Proposition~\ref{lem:distribution-normalization}.
The description of \(K_i^+\) is its defining quadratic-form condition.
Applying the defining inequalities to \(\chi\) in \eqref{eq:inner-distribution} and \eqref{eq:dual-distribution} and using \(h_A(\chi^2)>0\) gives nonnegativity of the distributions.
For a feasible pair, \(\mathsf B=Q'\mathsf A\) and \(E_0=N^{-1}J\) give the objective identities.
Their common value is nonnegative because \(E_0\in K_0^+\).
It remains to show positivity of \(\mathsf A\) on positive operators.
If \(w\in W\) is positive on \(A\), each central orthogonal projection \(e_i\) commutes with \(w\), so \(e_i\circ w\in K_i^+\).
The inverse MacWilliams relation therefore gives
\[
\mathsf A(w)=\frac1N\sum_{i\in\mathcal S}\mathsf B_i(e_i\circ w)\geq0.
\]
\end{proof}

For actual distributions, both objectives equal effective size by Proposition~\ref{lem:distribution-normalization}.
We maximize this value for codes and minimize it for designs.

\begin{definition}
For \(0\in\Delta\subseteq\mathcal C\), let \(\DelsCode(W,\Delta)\) be the supremum of
\begin{equation}\label{def:code-conic-bound}
\begin{array}{ll}
\text{maximize} & \mathsf A(J),\\
\text{subject to} & \text{the conditions in }\eqref{eq:conic-feasibility},\\
& \mathsf A_a=0\quad(a\notin\Delta).
\end{array}
\end{equation}
For \(T\subseteq\mathcal S\setminus\{0\}\), let \(\DelsDes(W,T)\) be the infimum of
\begin{equation}\label{def:design-conic-bound}
\begin{array}{ll}
\text{minimize} & \mathsf B_0(E_0),\\
\text{subject to} & \text{the conditions in }\eqref{eq:conic-feasibility},\\
& \mathsf B_i=0\quad(i\in T).
\end{array}
\end{equation}
The variables are \((\mathsf A,\mathsf B)\), and we use \(\sup\varnothing=-\infty\) and \(\inf\varnothing=+\infty\).
\end{definition}

\begin{theorem}\label{thm:conic-bounds}
Let \(0\in\Delta\subseteq\mathcal C\) and \(T\subseteq\mathcal S\setminus\{0\}\).
Every \(\Delta\)-code \(\chi\in A_+\setminus\{0\}\) satisfies
\[
\esize(\chi)\leq\DelsCode(W,\Delta),
\]
and every \(T\)-design \(\chi\in A_+\setminus\{0\}\) satisfies
\[
\esize(\chi)\geq\DelsDes(W,T).
\]
\end{theorem}

\begin{proof}
It suffices to use the actual distributions as feasible points.
Proposition~\ref{lem:distribution-normalization}, Theorem~\ref{thm:macwilliams}, and Proposition~\ref{prop:distribution-positivity} show that \((A^\chi,B^\chi)\) satisfies \eqref{eq:conic-feasibility}.
If \(\chi\) is a \(\Delta\)-code, its distribution also satisfies the additional constraints in \eqref{def:code-conic-bound}; its objective value \(\esize(\chi)\) is therefore at most the supremum.
If \(\chi\) is a \(T\)-design, the same pair satisfies \eqref{def:design-conic-bound}, and its objective value is at least the infimum.
\end{proof}

A feasible pair need not arise from a positive element, so a code or design need not attain the bound.
When \(\dim U_0>1\), the objective \(\mathsf B_0(E_0)\) specifies only one evaluation of \(\mathsf B_0\).

We next identify the cones in the classical setting.
For homogeneous coherent configurations, the ideals for the Schur product are one-dimensional, while those for composition can be matrix algebras.
The inner variables are then scalars and the dual variables are positive semidefinite matrices.

\begin{proposition}\label{prop:classical-noncommutative-specialization}
Let \(W\) be the adjacency algebra of a homogeneous coherent configuration on a finite set \(X\), with adjacency basis \(\{B_a\}_{a\in\mathcal C}\).
Regard \(W\) as a QAS on \(A=\C^X\), with the simple-ideal decompositions above.
Then
\[
V_a=\C B_a,\qquad R_a^+=\R_{\geq0}B_a.
\]
For a nonempty \(Y\subseteq X\) with characteristic vector \(\chi_Y\),
\[
A_a^{\chi_Y}(B_a)=\frac1{|Y|}\chi_Y^\dagger B_a\chi_Y,\qquad
B_i^{\chi_Y}(u)=\frac{|X|}{|Y|}\chi_Y^\dagger u\chi_Y
\quad(u\in U_i),
\]
and
\[
\esize(\chi_Y)=|Y|,\qquad
B_i^{\chi_Y}=0\quad\Longleftrightarrow\quad e_i(\chi_Y)=0.
\]
After identifying \((U_i,\circ,\dagger)\) with \(M_r(\C)\) by a \(C^\ast\)-algebra isomorphism, its complex linear functionals nonnegative on \(K_i^+\) are precisely
\[
u\longmapsto\operatorname{Tr}(Du),\qquad D\in M_r(\C),\quad D\succeq0.
\]
The code and design bounds therefore become semidefinite programs with scalar variables for the \(V_a\) and positive semidefinite matrix variables for the \(U_i\).
\end{proposition}

\begin{proof}
We first identify the ideals and cones for the Schur product.
The adjacency algebra is a cocommutative QAS~\cite{Higman1970,Gromada2024}.
The \(B_a\) are minimal projections for the entrywise Schur product, so \(V_a=\C B_a\).
Their entries are nonnegative, giving \(\R_{\geq0}B_a\subseteq R_a^+\).
Conversely, if \(cB_a\in R_a^+\), self-adjointness for \(\ast\) gives \(c\in\R\), and testing on the constant function with value \(1\) gives
\[
c\sum_{x,y\in X}(B_a)_{xy}\geq0.
\]
The sum is positive because \(B_a\) is a nonzero \(0,1\)-matrix; hence \(c\geq0\).

Substituting \(N=|X|\) and \(h_A(\chi_Y^2)=|Y|\) in \eqref{eq:inner-distribution}--\eqref{eq:dual-distribution} gives the distribution formulas.
Proposition~\ref{lem:distribution-normalization} gives \(\esize(\chi_Y)=h_A(\chi_Y)=|Y|\), and Lemma~\ref{lem:spectral-vanishing} gives the vanishing criterion.
For composition, the \(C^\ast\)-algebra isomorphism identifies \(K_i^+\) with the positive semidefinite cone in \(M_r(\C)\).
Self-duality under the trace pairing gives the stated representation of its positive functionals~\cite[Section~2.4]{BoydVandenberghe2004}.
The cone constraints are therefore scalar nonnegativity and positive semidefiniteness; all other constraints and both objectives are linear.
\end{proof}

For symmetric association schemes, both sets of ideals are one-dimensional, and we recover the programs of Section~\ref{sec:classical-delsarte}.

\begin{corollary}\label{cor:classical-programs}
For the Bose--Mesner algebra of a symmetric association scheme, the conic programs~\eqref{def:code-conic-bound} and~\eqref{def:design-conic-bound} reduce to the linear programs~\eqref{def:classical-code-program} and~\eqref{def:classical-design-program}, respectively.
\end{corollary}

\begin{proof}
Proposition~\ref{prop:classical-noncommutative-specialization} gives \(V_a=\C B_a\) and \(R_a^+=\R_{\geq0}B_a\).
Symmetry also gives \(U_i=\C E_i\) and \(K_i^+=\R_{\geq0}E_i\).
Thus the functionals are determined by real scalars \(a_i:=\mathsf A_i(B_i)\) and \(b_j:=\mathsf B_j(E_j)\), and the cone constraints are \(a_i,b_j\geq0\).
Evaluating the MacWilliams relation on \(E_j\) and \(B_i\), using \eqref{eq:classical-eigenmatrices}, gives the eigenmatrix relations in Proposition~\ref{prop:classical-distributions}.
Since \(I=B_0=\sum_j E_j\), the common normalization is \(a_0=1\), equivalently \(v^{-1}\sum_jb_j=1\).
The objectives are \(\mathsf A(J)=\sum_i a_i\) and \(\mathsf B_0(E_0)=b_0\), and the respective code and design conditions set exactly the prescribed components to zero.
\end{proof}

For a general QAS, semidefinite representations of the cones \(R_a^+\) need not be known.
The bounds are finite-dimensional conic programs; they are semidefinite programs when the required cone constraints admit such representations~\cite[Chapters~2 and~5]{BoydVandenberghe2004}.
\section{Moore graphs as codes and designs}\label{sec:moore}

A finite simple \(k\)-regular graph of diameter two on \(k^2+1\) vertices is a Moore graph~\cite{HoffmanSingleton1960}.
We represent each vertex and its neighborhood by \((x,B)\).
The graph conditions then become code and design conditions in a QAS whose composition algebra contains \(M_2(\C)\).

Let a finite group \(G\) act transitively on \(\Omega\), with
\[
(g\cdot f)(\omega):=f(g^{-1}\omega)
\qquad(f\in\C^\Omega).
\]
Define
\[
\End_G(\C^\Omega):=
\left\{S\in\End_{\C}(\C^\Omega):
\begin{array}{l}
S(g\cdot f)=g\cdot S(f)
\\
(g\in G,\ f\in\C^\Omega)
\end{array}
\right\}.
\]
For each \(G\)-orbit \(R\subseteq\Omega^2\), its adjacency matrix is
\[
(B_R)_{\omega,\eta}:=
\begin{cases}
1,&(\omega,\eta)\in R,\\
0,&(\omega,\eta)\notin R.
\end{cases}
\]
Commuting with the action is equivalent to having entries constant on each orbit, so
\[
W:=\End_G(\C^\Omega)
=\operatorname{span}_{\C}\{B_R:R\text{ is a }G\text{-orbit on }\Omega^2\}.
\]
These orbits form a homogeneous coherent configuration~\cite{Higman1970}, and \(W\) is a cocommutative QAS~\cite[Theorem~2.9]{Gromada2024}.
If
\[
\C^\Omega\cong\bigoplus_\lambda V_\lambda\otimes\C^{m_\lambda},
\]
where the \(V_\lambda\) are pairwise nonisomorphic irreducible \(G\)-modules, Schur's lemma gives~\cite{Serre1977}
\begin{equation}\label{eq:schurian-commutant}
(W,\circ,\dagger)\cong\bigoplus_\lambda M_{m_\lambda}(\C).
\end{equation}
Thus \(W\) is noncommutative for composition if and only if some \(m_\lambda\geq2\).
Let \(U_i\) (\(i\in\mathcal S\)) be its simple ideals for composition, with central orthogonal projections \(e_i\), indexed so that \(E_0\in U_0\).
Transitivity gives multiplicity one for the trivial representation, so
\[
U_0=\C E_0,\qquad e_0=E_0.
\]
For \(Y\subseteq\Omega\), put
\[
\chi_Y(\omega):=
\begin{cases}
1,&\omega\in Y,\\
0,&\omega\notin Y,
\end{cases}
\qquad 1:=\chi_\Omega.
\]

The following proposition uses orthogonality to bound design size and determine intersection sizes.

\begin{proposition}\label{prop:schurian-design-orthogonality-bound}
Suppose that \(T\subseteq\mathcal S\setminus\{0\}\) and \(\varnothing\ne Z\subseteq\Omega\) satisfy
\[
e_i\chi_Z=0\qquad(i\notin T\cup\{0\}).
\]
Then
\[
\DelsDes(W,T)\geq\frac{|\Omega|}{|Z|}.
\]
Every \(T\)-design \(Y\subseteq\Omega\) satisfies
\[
|Y|\geq\frac{|\Omega|}{|Z|},\qquad
|Y\cap gZ|=\frac{|Y||Z|}{|\Omega|}\quad(g\in G),
\]
and hence \(|\Omega|\mid |Y||Z|\).
\end{proposition}

\begin{proof}
To bound the design program, average the rank-one operators associated with the translates of \(Z\):
\[
M:=\frac{|\Omega|}{|Z||G|}
\sum_{g\in G}\chi_{gZ}\chi_{gZ}^{\dagger}\in W.
\]
By transitivity, each point belongs to \(|G||Z|/|\Omega|\) translates, counted with multiplicity.
Thus \(M\) has diagonal entries one, nonnegative entries, and \(M1=|Z|1\).
In particular, \(M-I\) is a nonnegative linear combination of nondiagonal adjacency matrices.
Since each \(e_i\) commutes with the group action, the hypothesis also gives \(e_i\chi_{gZ}=0\) for \(i\notin T\cup\{0\}\).
Consequently,
\[
e_0\circ M=|Z|E_0,\qquad
e_i\circ M=0\quad(i\notin T\cup\{0\}).
\]
For any feasible pair \((\mathsf A,\mathsf B)\) in the design program, Proposition~\ref{prop:classical-noncommutative-specialization} gives nonnegativity of \(\mathsf A\) on adjacency matrices.
Using \(\mathsf B_i=0\) for \(i\in T\) and the inverse MacWilliams relation, we obtain
\[
1=\mathsf A(I)\leq\mathsf A(M)
=\frac1{|\Omega|}\sum_{i\in\mathcal S}\mathsf B_i(e_i\circ M)
=\frac{|Z|}{|\Omega|}\mathsf B_0(E_0).
\]
Taking the infimum proves the program bound; Theorem~\ref{thm:conic-bounds} then gives the bound on \(|Y|\).

For the intersection formula, Lemma~\ref{lem:spectral-vanishing} gives \(e_i\chi_Y=0\) for \(i\in T\), while \(e_i\chi_{gZ}=0\) for \(i\notin T\cup\{0\}\).
The orthogonal decomposition by the \(e_i\) therefore gives
\[
|Y\cap gZ|
=\langle\chi_Y,\chi_{gZ}\rangle
=\langle E_0\chi_Y,E_0\chi_{gZ}\rangle
=\frac{|Y||Z|}{|\Omega|}.
\]
The left-hand side is an integer, proving divisibility.
This intersection formula also appears in~\cite[Definition~5.3 and Proposition~5.5]{BambergLansdown2026}.
\end{proof}

Let \(k\geq2\) be an integer, put \(v:=k^2+1\), and let \(X\) be a set of size \(v\).
The group \(G:=\operatorname{Sym}(X)\) acts transitively on
\[
\Omega:=\{(x,B):x\in X,\ B\subseteq X\setminus\{x\},\ |B|=k\}
\]
by \(g\cdot(x,B):=(g(x),g(B))\).
Set \(W:=\End_G(\C^\Omega)\).
For this \(W\), retain the notation \(U_i,e_i\) (\(i\in\mathcal S\)) above.
To describe the occurrences of points as first coordinates and in the subsets, use the standard representation
\[
V:=\left\{(a_x)_{x\in X}\in\C^X:\sum_{x\in X}a_x=0\right\}
\]
and the \(G\)-equivariant maps \(\Phi,\Psi:V\to\C^\Omega\) defined by
\[
\Phi(a)(x,B):=a_x,\qquad
\Psi(a)(x,B):=\sum_{z\in B}a_z.
\]

\begin{proposition}\label{prop:moore-standard-block}
The multiplicity of \(V\) in \(\C^\Omega\) is two, and the sum of all submodules isomorphic to \(V\) is \(\Phi(V)\oplus\Psi(V)\).
Put
\[
T:=\{i\in\mathcal S\setminus\{0\}:
\operatorname{im}e_i=\Phi(V)\oplus\Psi(V)\}.
\]
Then
\[
|T|=1,\qquad U_i\cong M_2(\C)\quad(i\in T).
\]
For every nonempty \(Y\subseteq\Omega\),
\[
\begin{aligned}
Y\text{ is a }T\text{-design}
&\quad\Longleftrightarrow\quad
\chi_Y\in\bigl(\Phi(V)\oplus\Psi(V)\bigr)^\perp\\
&\quad\Longleftrightarrow\quad
\begin{cases}
\displaystyle |\{B:(x,B)\in Y\}|=\frac{|Y|}{v},\\[5pt]
\displaystyle |\{(y,B)\in Y:x\in B\}|=\frac{k|Y|}{v}
\end{cases}
\quad(\forall x\in X).
\end{aligned}
\]
\end{proposition}

\begin{proof}
The stabilizer of \((x,B)\in\Omega\) has three orbits on \(X\):
\[
\{x\},\qquad B,\qquad X\setminus(\{x\}\cup B).
\]
Its fixed vectors in \(V\) are constant on these orbits and have coordinate sum zero, so their space has dimension two.
Frobenius reciprocity gives multiplicity two for \(V\) in \(\C^\Omega\)~\cite{Serre1977}.
The maps \(\Phi\) and \(\Psi\) are nonzero and not proportional: choose \(p\in B\), \(q\notin B\cup\{x\}\), and \(a\in V\) with \(a_p=1\), \(a_q=-1\), and all other coordinates zero.
Then
\[
\Phi(a)(x,B)=0,\qquad \Psi(a)(x,B)=1.
\]
By irreducibility of \(V\) and Schur's lemma, their images are distinct irreducible submodules with zero intersection.
They account for both copies of \(V\), and~\eqref{eq:schurian-commutant} gives \(|T|=1\) and the corresponding ideal \(M_2(\C)\).

Lemma~\ref{lem:spectral-vanishing} gives the first equivalence.
For \(a\in V\),
\[
\begin{aligned}
\langle\chi_Y,\Phi(a)\rangle
&=\sum_{x\in X}|\{B:(x,B)\in Y\}|a_x,\\
\langle\chi_Y,\Psi(a)\rangle
&=\sum_{z\in X}|\{(x,B)\in Y:z\in B\}|a_z.
\end{aligned}
\]
Each expression vanishes for every \(a\) with coordinate sum zero if and only if its count vector is constant.
Summing the two count vectors over \(X\) gives \(|Y|\) and \(k|Y|\), respectively, yielding the stated values.
\end{proof}

To impose adjacency and common-neighbor conditions, index the \(G\)-orbits on \(\Omega^2\) by \(R_a\) (\(a\in\mathcal C\)), with
\[
R_0:=\{(\omega,\omega):\omega\in\Omega\}.
\]
Let \(R\subseteq\Omega^2\) consist of the pairs \(((x,B),(y,C))\) with \(x\ne y\) satisfying
\[
\begin{aligned}
&y\in B,\quad x\in C,\quad B\cap C=\varnothing,\\
\text{or }\quad
&y\notin B,\quad x\notin C,\quad |B\cap C|=1.
\end{aligned}
\]
These cases describe adjacent vertices with no common neighbor and nonadjacent vertices with exactly one common neighbor.
Since \(R\) is \(G\)-invariant, put
\[
\Delta:=\{a\in\mathcal C:R_a\subseteq R_0\cup R\}.
\]
By Proposition~\ref{prop:classical-noncommutative-specialization}, for nonempty \(Y\subseteq\Omega\),
\[
Y\text{ is a }\Delta\text{-code}
\quad\Longleftrightarrow\quad
Y\times Y\subseteq R_0\cup R.
\]

\begin{proposition}\label{prop:moore-correspondence}
For a finite simple \(k\)-regular graph \(\Gamma\) on \(X\), let \(N_\Gamma(x)\) be the neighborhood of \(x\) and put
\[
Y_\Gamma:=\{(x,N_\Gamma(x)):x\in X\}.
\]
The map \(\Gamma\mapsto Y_\Gamma\) is a bijection between Moore graphs of degree \(k\) and diameter two on \(X\) and nonempty subsets of \(\Omega\) that are both \(\Delta\)-codes and \(T\)-designs.
Every such subset has size \(v\).
\end{proposition}

\begin{proof}
Let \(Y\) be a nonempty \(\Delta\)-code and \(T\)-design.
The code condition allows at most one pair with any first coordinate, while Proposition~\ref{prop:moore-standard-block} gives
\[
0<\frac{|Y|}{v}=|\{B:(x,B)\in Y\}|\leq1
\qquad(x\in X).
\]
These counts are integers, so \(|Y|=v\) and
\[
Y=\{(x,B_x):x\in X\},\qquad
B_x\subseteq X\setminus\{x\},\quad |B_x|=k,
\]
with each \(B_x\) uniquely determined.
Define adjacency by \(x\sim y\) if and only if \(y\in B_x\).
The definition of \(R\) gives, for \(x\ne y\),
\[
y\in B_x\quad\Longleftrightarrow\quad x\in B_y,
\qquad
|B_x\cap B_y|=
\begin{cases}
0,& y\in B_x,\\
1,& y\notin B_x.
\end{cases}
\]
Thus the graph is simple and \(k\)-regular, and every nonadjacent pair has distance two.
Such pairs exist because \(k<v-1\), so the diameter is two and the graph has \(v=k^2+1\) vertices.

Conversely, let \(\Gamma\) be a Moore graph on \(X\).
From each \(x\in X\), the bound
\[
1+k+k(k-1)=v
\]
on the number of vertices reachable within two steps is attained.
Hence the \(k(k-1)\) walks of length two that do not immediately return reach distinct vertices outside \(\{x\}\cup N_\Gamma(x)\).
Adjacent vertices therefore have no common neighbor, and nonadjacent vertices have exactly one~\cite{HoffmanSingleton1960}.
This proves that \(Y_\Gamma\) is a \(\Delta\)-code.
Every first coordinate occurs once, and \(k\)-regularity makes every point belong to exactly \(k\) neighborhoods.
Proposition~\ref{prop:moore-standard-block} gives the design condition.
The constructions are inverse.
\end{proof}

The design bound is attained for every \(k\geq2\), independently of the existence of a Moore graph.

\begin{proposition}
\[
\DelsDes(W,T)=v,\qquad \DelsCode(W,\Delta)\leq v.
\]
The code bound is an equality whenever a Moore graph of degree \(k\) and diameter two exists.
\end{proposition}

\begin{proof}
Put
\[
F_x:=\{(x,B)\in\Omega\}\quad(x\in X),\qquad
|F_x|=\binom{v-1}{k}=\frac{|\Omega|}{v}.
\]
Since \(\chi_{F_x}\in\C1\oplus\Phi(V)\), Proposition~\ref{prop:schurian-design-orthogonality-bound} gives
\[
\DelsDes(W,T)\geq\frac{|\Omega|}{|F_x|}=v.
\]
To attain this bound, identify \(X\) with \(\mathbb Z/v\mathbb Z\) and take
\[
Y:=\{(x,\{x+1,\ldots,x+k\}):x\in\mathbb Z/v\mathbb Z\}.
\]
Each first coordinate occurs once, and each point belongs to \(k\) of these subsets.
By Proposition~\ref{prop:moore-standard-block}, \(Y\) is a \(T\)-design of size \(v\), so Theorem~\ref{thm:conic-bounds} gives \(\DelsDes(W,T)\leq v\).

For the code bound, use the \(G\)-invariant operator
\[
L:=\sum_{x\in X}\chi_{F_x}\chi_{F_x}^{\dagger}\in W,
\]
whose entries are one exactly when the first coordinates agree.
The sets \(F_x\) partition \(\Omega\), so \(vL/|\Omega|\) is the orthogonal projection onto functions depending only on the first coordinate.
This space contains the constants, hence
\[
L-\frac{|\Omega|}{v}E_0\succeq0
\]
as an operator on \(\C^\Omega\).
Every feasible \(\mathsf A\) is nonnegative on positive operators by Proposition~\ref{prop:distribution-positivity}.
All nondiagonal orbits occurring in \(L\) lie outside \(\Delta\), so the code constraints give \(\mathsf A(L)=\mathsf A(I)=1\).
It follows that
\[
0\leq\mathsf A\left(L-\frac{|\Omega|}{v}E_0\right)
=1-\frac{|\Omega|}{v}\mathsf A(E_0)
=1-\frac{\mathsf A(J)}{v}.
\]
Taking the supremum proves the code bound.
If a Moore graph exists, Proposition~\ref{prop:moore-correspondence} gives a code of size \(v\), and Theorem~\ref{thm:conic-bounds} gives equality.
\end{proof}

For \(k=3\), the Petersen graph on \(X:=\binom{\{1,2,3,4,5\}}2\), with disjoint pairs adjacent, is a Moore graph~\cite{HoffmanSingleton1960}.
Its subset \(Y_\Gamma\) is a code and a design of size ten, attaining
\[
\DelsCode(W,\Delta)=\DelsDes(W,T)=10.
\]
\section{Examples on matrix algebras}\label{sec:quantum-examples}

\subsection{Translation Hamming schemes}\label{sec:translation-hamming}

Weyl operators realize the Hamming scheme \(H(n,q^2)\) on a matrix algebra.
Its code and design conditions become vanishing error compressions and scalar partial traces, respectively.
We use these descriptions to obtain bounds for pure quantum codes and quantum orthogonal arrays.

Let \(q\) be a prime power and \(n\) a positive integer.
Put \(\mathcal H:=(\C^q)^{\otimes n}\), \(D:=q^n\), and
\[
A:=\End(\mathcal H)\cong M_D(\C),\qquad h_A(X):=D\operatorname{Tr}(X).
\]
Thus the dimension of the finite quantum space is \(h_A(1_A)=D^2\).

Fix a nontrivial additive character \(\psi\) of \(\mathbb F_q\), and let \(\{e_x:x\in\mathbb F_q\}\) be the standard orthonormal basis of \(\C^q\).
Define the Weyl operators~\cite{Knill1996,AshikhminKnill2001} by
\[
W_{(a,b)}e_x:=\psi(bx)e_{x+a},\qquad
W_\alpha:=W_{\alpha_1}\otimes\cdots\otimes W_{\alpha_n}
\quad\bigl(\alpha\in G:=(\mathbb F_q^2)^n\bigr).
\]
The weight records the number of tensor factors on which a Weyl operator acts nontrivially:
\[
\operatorname{supp}(\alpha):=\{j\in[n]:\alpha_j\ne(0,0)\},\qquad
\operatorname{wt}(\alpha):=|\operatorname{supp}(\alpha)|.
\]
For \(\alpha_j=(a_j,b_j)\) and \(\beta_j=(c_j,d_j)\), put \([\alpha,\beta]:=\sum_{j=1}^n(b_jc_j-a_jd_j)\).

The Weyl operators are unitary and satisfy
\[
\begin{aligned}
\operatorname{Tr}(W_\alpha^\ast W_\beta)&=D\delta_{\alpha,\beta},\\
W_\alpha W_\beta&=\psi\!\left(\sum_{j=1}^n b_jc_j\right)W_{\alpha+\beta},\\
W_\alpha W_\beta W_\alpha^\ast&=\psi([\alpha,\beta])W_\beta.
\end{aligned}
\]
The defining action gives
\[
W_{(a,b)}W_{(c,d)}e_x=\psi((b+d)x+bc)e_{x+a+c}.
\]
Exchanging the factors proves the conjugation identity.
Different translations have zero trace, and character orthogonality gives
\[
\operatorname{Tr}(W_{(a,b)}^\ast W_{(a,d)})
=\sum_{x\in\mathbb F_q}\psi((d-b)x)=q\delta_{b,d}.
\]
Taking tensor products proves trace orthogonality.
Thus the \(D^2=\dim_\C A\) Weyl operators form a basis of \(A\).

To group errors by weight, put \(G_r:=\{\alpha\in G:\operatorname{wt}(\alpha)=r\}\) for \(0\le r\le n\), and define
\[
\begin{aligned}
B_r(X)&:=\sum_{\alpha\in G_r}W_\alpha XW_\alpha^\ast,\\
E_s(X)&:=\frac1D\sum_{\beta\in G_s}\operatorname{Tr}(W_\beta^\ast X)W_\beta
\end{aligned}
\qquad(0\le r,s\le n,\ X\in A).
\]
Here \(E_s\) is the orthogonal projection onto the Weyl operators of weight \(s\).
Put
\[
W:=\operatorname{span}_{\C}\{B_0,\ldots,B_n\}.
\]
The change of basis is given by the Krawtchouk polynomials
\[
\mathsf K_r(s):=\sum_{\ell=0}^r(-1)^\ell(q^2-1)^{r-\ell}
\binom{s}{\ell}\binom{n-s}{r-\ell}
\qquad(0\le r,s\le n),
\]
where \(\binom{u}{v}=0\) for \(v<0\) or \(v>u\).

\begin{proposition}\label{prop:weyl-hamming-qas}
The space \(W\) is a commutative and cocommutative QAS on \(A\), with simple ideals
\[
V_r=\C B_r,\qquad U_s=\C E_s
\qquad(0\le r,s\le n)
\]
for \(\bullet\) and \(\circ\), respectively.
The identities \(B_0=I\) and \(E_0=D^{-2}J\) hold, and the two bases are related by
\[
\begin{aligned}
B_r&=\sum_{s=0}^n\mathsf K_r(s)E_s,\\
E_s&=\frac1{D^2}\sum_{r=0}^n\mathsf K_s(r)B_r.
\end{aligned}
\]
Thus its first and second eigenmatrices are those of \(H(n,q^2)\).
\end{proposition}

\begin{proof}
The conjugation identity gives
\[
B_r(W_\beta)=\left(\sum_{\alpha\in G_r}\psi([\alpha,\beta])\right)W_\beta.
\]
Character orthogonality in one coordinate gives
\[
\sum_{(a,b)\ne(0,0)}\psi(bc-ad)=
\begin{cases}
q^2-1,&(c,d)=(0,0),\\
-1,&(c,d)\ne(0,0).
\end{cases}
\]
If \(\operatorname{wt}(\beta)=s\), choosing \(\ell\) nonzero coordinates of \(\alpha\) inside \(\operatorname{supp}(\beta)\) and \(r-\ell\) outside gives
\[
\sum_{\alpha\in G_r}\psi([\alpha,\beta])=\mathsf K_r(s).
\]
This proves the first expansion; Krawtchouk inversion gives the second~\cite[Chapter~5]{Delsarte1973}.
Since the \(E_s\) are mutually orthogonal projections summing to \(I\), these expansions prove closure and commutativity under composition.

For standard matrix units \(\varepsilon_{ij}\) of \(M_D(\C)\),
\[
m^\dagger(\varepsilon_{ij})
=\frac1D\sum_{k=1}^D\varepsilon_{ik}\otimes\varepsilon_{kj}.
\]
Substitution gives, for unitary \(U,V\in A\),
\[
\bigl(X\mapsto UXU^\ast\bigr)\mathbin{\bullet}
\bigl(X\mapsto VXV^\ast\bigr)
=\frac{\operatorname{Tr}(U^\ast V)}D
\bigl(X\mapsto UXV^\ast\bigr).
\]
Weyl orthogonality therefore gives
\[
B_r\mathbin{\bullet}B_s=\delta_{r,s}B_r,\qquad B_r^\ast=B_r.
\]
Also \(B_r^\dagger=B_r\), since \(G_r=-G_r\) and \(W_{-\alpha}\) is a scalar multiple of \(W_\alpha^\ast\).
The definitions and the second expansion at \(s=0\) give
\[
B_0=I,\qquad E_0=D^{-2}J,\qquad \sum_{r=0}^nB_r=J.
\]
Thus \(W\) is a cocommutative quantum coherent algebra, with the stated simple ideals.
It is homogeneous by~\cite[Proposition~2.10]{Gromada2024}, since \(A\) is a matrix algebra.
The eigenmatrices follow from the two expansions.
\end{proof}

The Weyl product formula realizes the bicharacter twist of \(H(n,q^2)\) with
\[
(\alpha,\beta)\longmapsto
\psi\!\left(-\sum_{j=1}^n b_jc_j\right)
\qquad(\alpha_j=(a_j,b_j),\ \beta_j=(c_j,d_j))
\]
in the convention of~\cite[Section~4.3]{Gromada2024}.
The following cones reduce its conic bounds to linear programs.

\begin{lemma}\label{lem:weyl-hamming-cones}
For \(0\le r,s\le n\), the cones are
\[
R_r^+=\R_{\ge0}B_r,\qquad K_s^+=\R_{\ge0}E_s.
\]
\end{lemma}

\begin{proof}
For \(\chi\in A_+\), cyclicity of the trace gives
\[
\langle\chi,B_r\chi\rangle
=D\sum_{\alpha\in G_r}\operatorname{Tr}\!\left(
(\chi^{1/2}W_\alpha\chi^{1/2})
(\chi^{1/2}W_\alpha\chi^{1/2})^\ast\right)\ge0.
\]
Hence every nonnegative multiple of \(B_r\) belongs to \(R_r^+\).
Conversely, \(V_r=\C B_r\) and \(B_r^\ast=B_r\) imply that every self-adjoint element of \(V_r\) is a real multiple of \(B_r\).
Evaluating at \(1_A\) gives \(\langle1_A,B_r1_A\rangle=D^2|G_r|>0\), so a negative multiple cannot belong to \(R_r^+\).
The other equality follows from \(U_s=\C E_s\), since \(E_s\) is a nonzero orthogonal projection.
\end{proof}

\subsubsection{Quantum codes}

Error detection is expressed by compression to a code subspace~\cite[Section~III]{KnillLaflamme1997}.

\begin{definition}\label{def:quantum-code-purity}
A quantum code is a nonzero subspace of \(\mathcal H\), represented by its orthogonal projection \(\Pi\).
For \(1\le\delta\le n+1\), it detects all errors of weight less than \(\delta\) if
\[
\Pi W_\alpha\Pi=c_\alpha\Pi
\qquad(0<\operatorname{wt}(\alpha)<\delta)
\]
for scalars \(c_\alpha\in\C\).
It is pure to distance \(\delta\) if
\[
\Pi W_\alpha\Pi=0
\qquad(0<\operatorname{wt}(\alpha)<\delta).
\]
\end{definition}

These conventions agree with~\cite[Sections~1 and~3]{Rains1998} and~\cite{Rains1999}.
The following proposition identifies purity with the QAS code condition for
\[
\Delta:=\{0\}\cup\{r\in[n]:r\ge\delta\}.
\]

\begin{proposition}\label{prop:weyl-code-distributions}
Let \(\Pi\in A\) be a nonzero projection of rank \(K\).
Write its distributions as \(a_r:=A_r^\Pi(B_r)\) and \(b_s:=B_s^\Pi(E_s)\).
Then
\[
\begin{aligned}
a_r&=\frac1K\sum_{\alpha\in G_r}\operatorname{Tr}(\Pi W_\alpha\Pi W_\alpha^\ast),\\
b_s&=\frac DK\sum_{\beta\in G_s}|\operatorname{Tr}(\Pi W_\beta)|^2,
\end{aligned}
\]
and
\[
\begin{aligned}
b_s&=\sum_{r=0}^n\mathsf K_s(r)a_r,\\
a_r&=\frac1{D^2}\sum_{s=0}^n\mathsf K_r(s)b_s.
\end{aligned}
\]
Their normalizations are
\[
a_0=1,\qquad \sum_{r=0}^na_r=b_0=DK,\qquad \sum_{s=0}^nb_s=D^2.
\]
The projection \(\Pi\) is a \(\Delta\)-code if and only if
\begin{equation}\label{eq:weyl-code-condition}
\Pi W_\alpha\Pi=0
\qquad(0<\operatorname{wt}(\alpha)<\delta).
\end{equation}
\end{proposition}

\begin{proof}
Since \(h_A(\Pi^2)=DK\), substituting the formulas for \(B_r\) and \(E_s\) into equations~\eqref{eq:inner-distribution} and~\eqref{eq:dual-distribution} gives the distributions.
Their transforms follow from Proposition~\ref{prop:weyl-hamming-qas} and Theorem~\ref{thm:macwilliams}; their normalizations follow from Proposition~\ref{lem:distribution-normalization}, using \(B_0=I\), \(\sum_rB_r=J\), and \(\sum_sE_s=I\).
To identify the code condition, write each contribution as the trace of a positive operator:
\[
\operatorname{Tr}(\Pi W_\alpha\Pi W_\alpha^\ast)
=\operatorname{Tr}\!\left((\Pi W_\alpha\Pi)(\Pi W_\alpha\Pi)^\ast\right)\ge0.
\]
The sum defining \(a_r\) therefore vanishes exactly when every \(\Pi W_\alpha\Pi\), \(\alpha\in G_r\), is zero.
Because \(V_r=\C B_r\), the scalar equation \(a_r=0\) is equivalent to vanishing of the whole functional \(A_r^\Pi\).
Requiring this for \(1\le r<\delta\) is precisely the \(\Delta\)-code condition.
\end{proof}

Lemma~\ref{lem:weyl-hamming-cones} makes both positivity constraints scalar nonnegativity conditions.
Using the transform in Proposition~\ref{prop:weyl-code-distributions}, the code bound \(\DelsCode(W,\Delta)\) is therefore the optimum of
\[
\begin{array}{ll}
\mathrm{maximize} & \displaystyle\sum_{r=0}^na_r,\\[1mm]
\mathrm{subject\ to} & a_0=1,\\
& a_r=0 \quad (1\le r<\delta),\\
& a_r\ge0 \quad (\delta\le r\le n),\\
& \displaystyle\sum_{r=0}^n\mathsf K_s(r)a_r\ge0 \quad (1\le s\le n).
\end{array}
\]
This is the Delsarte linear program for codes of minimum distance at least \(\delta\) in \(H(n,q^2)\)~\cite[Chapter~5]{Delsarte1973}.
Since \(\esize(\Pi)=DK\), Proposition~\ref{prop:weyl-code-distributions} and Theorem~\ref{thm:conic-bounds} give, for every quantum code pure to distance \(\delta\),
\[
D\operatorname{rank}\Pi\le\DelsCode(W,\Delta).
\]

The Shor--Laflamme--Rains bounds strengthen the distribution constraints for quantum codes by imposing, in the present normalization,
\[
b_r\le DK\,a_r\qquad(0\le r\le n).
\]
Indeed, Cauchy--Schwarz gives \(|\operatorname{Tr}(\Pi W_\alpha)|^2\le K\operatorname{Tr}(\Pi W_\alpha\Pi W_\alpha^\ast)\), and summing over \(G_r\) gives the inequality above~\cite{ShorLaflamme1997,Rains1998,Rains1999}.
These inequalities use that the positive element is a projection of rank \(K\).
Together with the further constraints for \(q=2\) in~\cite[Theorem~10]{Rains1999}, they give restrictions not imposed by the two QAS positivity cones alone.

\subsubsection{Quantum orthogonal arrays}

The design condition requires scalar partial traces.
For \(\Lambda\subseteq[n]\), put
\[
\Lambda^c:=[n]\setminus\Lambda,\qquad
I_\Lambda:=\bigotimes_{j\in\Lambda}I_q.
\]
The partial trace
\[
\operatorname{Tr}_{\Lambda^c}:A\longrightarrow
\bigotimes_{j\in\Lambda}M_q(\C)
\]
is characterized by
\[
\operatorname{Tr}\!\left(Y\operatorname{Tr}_{\Lambda^c}(X)\right)
=\operatorname{Tr}\!\left((Y\otimes I_{\Lambda^c})X\right)
\]
for \(X\in A\) and \(Y\in\bigotimes_{j\in\Lambda}M_q(\C)\)~\cite[Section~2.4]{NielsenChuang2010}.
Here the tensor factors are ordered as \(\Lambda,\Lambda^c\), with increasing indices within each group, and the empty tensor product is \(\C\).

Quantum orthogonal arrays prescribe uniform partial traces of row vectors.
We use the definition of Du et al.~\cite{DuYinPangWang2020} and Zang et al.~\cite[Definition~4.1]{ZangFacchiTian2021}, which includes cross terms.
The original definition uses only the sum of the row projections~\cite[Definition~8]{GoyenecheRaissiDiMartinoZyczkowski2018}.

\begin{definition}\label{def:qoa}
Let \(R\) be a positive integer and \(t\in\{0,\ldots,n\}\).
A quantum orthogonal array \(\operatorname{QOA}(R,n,q,t)\) is a family \((\varphi_j)_{j=1}^R\in\mathcal H^R\) with \(\|\varphi_j\|=1\) for \(j\in[R]\), satisfying
\[
\sum_{j,l=1}^R\operatorname{Tr}_{\Lambda^c}(\varphi_j\varphi_l^\dagger)
=\frac{R}{q^t}I_\Lambda
\qquad(\Lambda\subseteq[n],\ |\Lambda|=t).
\]
The integer \(t\) is called its strength.
\end{definition}

The definition applies to any integer \(q\ge2\). Here we use a finite-field Weyl realization, so we assume that \(q\) is a prime power.
The next theorem expresses the design condition by partial traces and applies it to the sum of the QOA rows.

\begin{theorem}\label{thm:qoa-qas-design}
Let \(t\in\{0,\ldots,n\}\) and \(\chi\in A_+\setminus\{0\}\).
Then \(\chi\) is a \(\{1,\ldots,t\}\)-design for \(W\) if and only if
\[
\operatorname{Tr}_{\Lambda^c}(\chi)
=\frac{\operatorname{Tr}\chi}{q^t}I_\Lambda
\qquad(\Lambda\subseteq[n],\ |\Lambda|=t).
\]
For \((\varphi_j)_{j=1}^R\in\mathcal H^R\) with \(\|\varphi_j\|=1\), put
\[
\varphi=\frac1{\sqrt R}\sum_{j=1}^R\varphi_j,
\qquad p=\varphi\varphi^\dagger.
\]
The family is a \(\operatorname{QOA}(R,n,q,t)\) if and only if \(\|\varphi\|=1\) and \(p\) is a \(\{1,\ldots,t\}\)-design.
For \(\|\varphi\|=1\),
\[
p\text{ is a }\{1,\ldots,t\}\text{-design}
\quad\Longleftrightarrow\quad
p\text{ is a }\{0\}\cup\{t+1,\ldots,n\}\text{-code}.
\]
In particular, every QOA determines a rank-one projection with \(\esize(p)=D\).
\end{theorem}

\begin{proof}
We first express each partial-trace condition in Weyl coefficients.
Fix \(\Lambda\subseteq[n]\) with \(|\Lambda|=t\).
The Weyl operators on these factors form an orthogonal basis, whose weight-zero element is the identity.
Its coefficient in \(\operatorname{Tr}_{\Lambda^c}(\chi)\) is fixed by \(\operatorname{Tr}\chi\).
By the defining property of the partial trace, the reduced operator is therefore the required scalar identity exactly when
\[
\operatorname{Tr}(\chi W_\beta^\ast)=0
\qquad(\varnothing\ne\operatorname{supp}(\beta)\subseteq\Lambda).
\]
Every nonempty support of size at most \(t\) is contained in such a \(\Lambda\).
Thus all the partial-trace conditions hold exactly when
\[
\operatorname{Tr}(\chi W_\beta^\ast)=0
\qquad(1\le\operatorname{wt}(\beta)\le t).
\]
Equation~\eqref{eq:dual-distribution} gives
\[
B_s^\chi(E_s)=\frac{D}{\operatorname{Tr}(\chi^2)}
\sum_{\beta\in G_s}|\operatorname{Tr}(\chi W_\beta^\ast)|^2.
\]
Each sum vanishes exactly when its Weyl coefficients vanish.
Since \(U_s=\C E_s\), this is equivalent to \(B_s^\chi=0\) for \(1\le s\le t\), proving the first assertion.

For the row vectors in the theorem, the cross terms are collected by \(p\).
Linearity of the partial trace gives
\[
\sum_{j,l=1}^R\operatorname{Tr}_{\Lambda^c}(\varphi_j\varphi_l^\dagger)
=R\operatorname{Tr}_{\Lambda^c}(p).
\]
Taking the full trace of a QOA identity gives \(R\|\varphi\|^2=R\), hence \(\|\varphi\|=1\).
With this normalization, the displayed identity is exactly the partial-trace condition for \(p\).
Conversely, that condition and \(\|\varphi\|=1\) give the QOA identities.

To compare the rank-one design with the code condition, observe that
\[
pW_\alpha p=(\varphi^\dagger W_\alpha\varphi)p
=\operatorname{Tr}(pW_\alpha)p.
\]
Since \(p=p^\dagger\), we have \(\operatorname{Tr}(pW_\alpha^\ast)=\overline{\operatorname{Tr}(pW_\alpha)}\).
Hence the Weyl coefficients of weights \(1,\ldots,t\) vanish if and only if \(pW_\alpha p=0\) for those weights.
Proposition~\ref{prop:weyl-code-distributions}, with \(\delta=t+1\), gives the code equivalence.
Finally, \(p\) is a rank-one projection and \(h_A(p)=D\), so \(\esize(p)=D\).
\end{proof}

This pure-state correspondence is known for \(1\le t\le\lfloor n/2\rfloor\)~\cite[Propositions~2 and~3]{Scott2004}.
Thus a QOA satisfies both the design lower bound and the code upper bound.

\begin{proposition}\label{prop:qoa-bound}
Let \(t\in\{0,\ldots,n\}\).
The value \(\DelsDes(W,\{1,\ldots,t\})\) is the optimum of
\[
\begin{array}{ll}
\mathrm{minimize} & b_0,\\[1mm]
\mathrm{subject\ to} & b_s=0 \quad (1\le s\le t),\\
& b_s\ge0 \quad (0\le s\le n),\\
& \displaystyle\frac1{D^2}\sum_{s=0}^n\mathsf K_r(s)b_s\ge0 \quad (1\le r\le n),\\[3mm]
& \displaystyle\frac1{D^2}\sum_{s=0}^nb_s=1.
\end{array}
\]
With the normalization discussed at the end of Section~\ref{sec:classical-delsarte}, this is the Delsarte linear program for \(\{1,\ldots,t\}\)-designs in \(H(n,q^2)\)~\cite[Chapter~5]{Delsarte1973}.
Consequently, the existence of a \(\operatorname{QOA}(R,n,q,t)\) implies
\[
\DelsDes(W,\{1,\ldots,t\})
\le D
\le\DelsCode(W,\{0\}\cup\{t+1,\ldots,n\}).
\]
The resulting constraints on \(n,q,t\) are independent of \(R\).
\end{proposition}

\begin{proof}
The cone description in Lemma~\ref{lem:weyl-hamming-cones} and the Krawtchouk transform turn the design program~\eqref{def:design-conic-bound} into the displayed linear program.
At \(r=0\), the transform equals the normalization and hence is already positive.
If a QOA exists, Theorem~\ref{thm:qoa-qas-design} supplies a rank-one projection satisfying both the design and code conditions, with effective size \(D\).
The two inequalities now follow from Theorem~\ref{thm:conic-bounds}.
\end{proof}

\begin{remark}
The two QOA definitions differ even for classical rows.
Let \(t\in\{0,\ldots,n\}\) and \(\varnothing\ne C\subseteq\mathbb F_q^n\).
Put \(|x\rangle:=e_{x_1}\otimes\cdots\otimes e_{x_n}\) and \(x_\Lambda:=(x_j)_{j\in\Lambda}\) for \(x\in C\) and \(\Lambda\subseteq[n]\).
The diagonal projection
\[
\Pi=\sum_{x\in C}|x\rangle\langle x|
\]
is a \(\{1,\ldots,t\}\)-design if and only if \(C\) is a classical orthogonal array of strength \(t\), since its partial traces are diagonal with entries
\[
\bigl(\operatorname{Tr}_{\Lambda^c}(\Pi)\bigr)_{u,u}
=\bigl|\{x\in C:x_\Lambda=u\}\bigr|,
\qquad(u\in\mathbb F_q^\Lambda).
\]
For \(|\Lambda|=t\), the design condition requires these counts to equal \(|C|/q^t\)~\cite[Chapter~1]{HedayatSloaneStufken1999}.

For the definition with cross terms, put
\[
\varphi=\frac1{\sqrt{|C|}}\sum_{x\in C}|x\rangle.
\]
Its partial trace is
\[
\operatorname{Tr}_{\Lambda^c}(\varphi\varphi^\dagger)
=\frac1{|C|}
\sum_{\substack{x,y\in C\\x_{\Lambda^c}=y_{\Lambda^c}}}
|x_\Lambda\rangle\langle y_\Lambda|.
\]
The diagonal entries give the same counts divided by \(|C|\).
Every matrix-entry contribution has coefficient \(1/|C|>0\), so an off-diagonal entry is nonzero exactly when two distinct rows agree outside \(\Lambda\).
Thus every off-diagonal entry vanishes for every \(|\Lambda|=t\) if and only if
\[
\bigl|\{j\in[n]:x_j\ne y_j\}\bigr|\ge t+1
\qquad(x,y\in C,\ x\ne y).
\]
Consequently, the basis vectors indexed by \(C\) form a QOA if and only if \(C\) is an orthogonal array of strength \(t\) and its distinct rows have Hamming distance at least \(t+1\).
The distance condition says that deleting any \(t\) columns leaves pairwise distinct rows~\cite[Section~IV]{GoyenecheZyczkowski2014}.
The conditions on the classical rows use \(H(n,q)\).
For the associated rank-one projection, the quantum conic bounds use the eigenmatrices of \(H(n,q^2)\).
\end{remark}

\subsection{A scheme on \(M_n(\C)\)}\label{sec:matrix-example}

We construct a QAS on \(M_n(\C)\) whose codes and designs are described by matrix entries.
Both products are noncommutative for \(n\ge2\), and its conic bounds can be computed exactly.
Let \(n\ge1\) and \(A:=M_n(\C)\), with matrix entries indexed by \([n]\).
For the standard matrix units \(\varepsilon_{ab}\), define
\[
\mathsf R_{ab,cd}(\varepsilon_{uv})
:=\delta_{u,c}\delta_{v,d}\varepsilon_{ab}
\qquad(a,b,c,d,u,v\in[n]),
\]
and set
\[
W_n:=\operatorname{span}_{\C}
\{\mathsf R_{ab,cd}:a,b,c,d\in[n],\ a-b=c-d\}.
\]
Differences are taken in \(\mathbb Z\).
Here
\[
h_A(x)=n\operatorname{Tr}(x),\qquad E_0=n^{-2}J.
\]

The following identities establish the QAS axioms and determine its simple ideals.

\begin{lemma}\label{lem:matrix-difference-basic}
For \(a,b,c,d,e,f,g,h\in[n]\),
\[
\begin{aligned}
\mathsf R_{ab,cd}\circ\mathsf R_{ef,gh}
&=\delta_{c,e}\delta_{d,f}\mathsf R_{ab,gh},\\
\mathsf R_{ab,cd}\mathbin{\bullet}\mathsf R_{ef,gh}
&=\frac1n\delta_{d,g}\delta_{b,e}\mathsf R_{af,ch}.
\end{aligned}
\]
Also,
\[
\mathsf R_{ab,cd}^{\dagger}=\mathsf R_{cd,ab},
\qquad
\mathsf R_{ab,cd}^{\ast}=\mathsf R_{ba,dc},
\]
\[
I=\sum_{a,b\in[n]}\mathsf R_{ab,ab},
\qquad
J=n\sum_{a,c\in[n]}\mathsf R_{aa,cc}.
\]
The space \(W_n\) is a QAS, and for \(n\ge2\) it is neither commutative nor cocommutative.
\end{lemma}

\begin{proof}
The composition and involution formulas, and the expressions for \(I,J\), follow by evaluation on the matrix units.
For the Schur product, substitute
\[
m^{\dagger}(\varepsilon_{uv})
=\frac1n\sum_{k\in[n]}\varepsilon_{uk}\otimes\varepsilon_{kv}
\]
into \(S\mathbin{\bullet}T=m(S\otimes T)m^{\dagger}\) and use \(\varepsilon_{ab}\varepsilon_{ef}=\delta_{b,e}\varepsilon_{af}\).
Each nonzero product satisfies the defining difference condition: for the Schur product, \(b=e\) and \(d=g\) give \(a-f=(a-b)+(e-f)=(c-d)+(g-h)=c-h\).
The formulas therefore show that \(W_n\) contains \(I,J\) and is closed under both products and involutions, so it is a quantum coherent algebra.
Every quantum coherent algebra on \(M_n(\C)\) is homogeneous by~\cite[Proposition~2.10]{Gromada2024}, hence \(W_n\) is a QAS.
For \(n\ge2\),
\[
\mathsf R_{11,22}\circ\mathsf R_{22,11}
=\mathsf R_{11,11}\ne\mathsf R_{22,22}
=\mathsf R_{22,11}\circ\mathsf R_{11,22},
\]
whereas
\[
\mathsf R_{12,12}\mathbin{\bullet}\mathsf R_{21,21}
=\frac1n\mathsf R_{11,11}\ne\frac1n\mathsf R_{22,22}
=\mathsf R_{21,21}\mathbin{\bullet}\mathsf R_{12,12}.
\qedhere
\]
\end{proof}

For composition and the quantum Schur product, group the basis elements by \(a-b\) and \(a-c\), respectively.
For an integer \(r\) with \(|r|<n\), put
\[
D_r:=\{i\in[n]:i-r\in[n]\},
\qquad |D_r|=n-|r|.
\]
For \(i,j\in D_r\), set
\[
\mathsf E_{ij}^{(r)}:=\mathsf R_{i,i-r,j,j-r},
\qquad
\mathsf F_{ij}^{(r)}:=\mathsf R_{i,j,i-r,j-r},
\]
and let
\[
U_r:=\operatorname{span}_{\C}\{\mathsf E_{ij}^{(r)}:i,j\in D_r\},
\qquad
V_r:=\operatorname{span}_{\C}\{\mathsf F_{ij}^{(r)}:i,j\in D_r\}.
\]
The defining equality \(a-b=c-d\) gives the unique expressions
\[
\mathsf R_{ab,cd}
=\mathsf E_{a,c}^{(a-b)}
=\mathsf F_{a,b}^{(a-c)}.
\]
By Lemma~\ref{lem:matrix-difference-basic}, products between distinct \(U_r\)'s vanish under \(\circ\), and products between distinct \(V_s\)'s vanish under \(\bullet\).
Within a summand,
\[
\begin{aligned}
\mathsf E_{ij}^{(r)}\circ\mathsf E_{k\ell}^{(r)}
&=\delta_{j,k}\mathsf E_{i\ell}^{(r)},
&
(\mathsf E_{ij}^{(r)})^\dagger&=\mathsf E_{ji}^{(r)},\\
\mathsf F_{ij}^{(r)}\mathbin{\bullet}\mathsf F_{k\ell}^{(r)}
&=\frac1n\delta_{j,k}\mathsf F_{i\ell}^{(r)},
&
(\mathsf F_{ij}^{(r)})^\ast&=\mathsf F_{ji}^{(r)}.
\end{aligned}
\]
Thus \(U_r\) and \(V_r\) are simple ideals, with matrix units \(\mathsf E_{ij}^{(r)}\) and \(n\mathsf F_{ij}^{(r)}\), respectively, and
\[
(W_n,\circ,\dagger)\cong\bigoplus_{|r|<n}M_{n-|r|}(\C),
\qquad
(W_n,\bullet,\ast)\cong\bigoplus_{|s|<n}M_{n-|s|}(\C).
\]
In particular, \(E_0\in U_0\) and \(I\in V_0\), since
\[
E_0=\frac1n\sum_{i,j\in[n]}\mathsf E_{ij}^{(0)},
\qquad
I=\sum_{i,j\in[n]}\mathsf F_{ij}^{(0)}.
\]

The two expressions for \(\mathsf R_{ab,cd}\) give the transition maps
\[
P_{r,s}(\mathsf F_{ij}^{(s)})=
\begin{cases}
\mathsf E_{i,i-s}^{(r)},&r=i-j,\\
0,&r\ne i-j
\end{cases}
\qquad(i,j\in D_s),
\]
\[
Q_{s,r}(\mathsf E_{ij}^{(r)})=
\begin{cases}
n^2\mathsf F_{i,i-r}^{(s)},&s=i-j,\\
0,&s\ne i-j
\end{cases}
\qquad(i,j\in D_r).
\]
The factor \(n^2=\dim A\) is the normalization of \(Q\).
The distributions on these matrix units determine the code and design conditions.

\begin{proposition}\label{prop:matrix-difference-distributions}
For \(\chi=(\chi_{ab})\in A_+\setminus\{0\}\), the distributions are
\[
\begin{aligned}
A_s^\chi(\mathsf F_{ij}^{(s)})
&=\frac{\overline{\chi_{ij}}\chi_{i-s,j-s}}{\operatorname{Tr}(\chi^2)}
&& (i,j\in D_s),\\
B_r^\chi(\mathsf E_{ij}^{(r)})
&=\frac{n^2\overline{\chi_{i,i-r}}\chi_{j,j-r}}{\operatorname{Tr}(\chi^2)}
&& (i,j\in D_r).
\end{aligned}
\]
\end{proposition}

\begin{proof}
The identities \(\mathsf F_{ij}^{(s)}(\chi)=\chi_{i-s,j-s}\varepsilon_{ij}\) and \(\mathsf E_{ij}^{(r)}(\chi)=\chi_{j,j-r}\varepsilon_{i,i-r}\) give both numerators.
Using \(h_A(\chi^2)=n\operatorname{Tr}(\chi^2)\) and \(N=n^2\) in \eqref{eq:inner-distribution} and \eqref{eq:dual-distribution} gives the stated formulas.
\end{proof}

Code conditions separate the nonzero diagonal entries; design conditions remove specified off-diagonal entries.
For integers \(1\le\delta\le n\) and \(0\le t\le n-1\), set
\[
\Delta_\delta:=\{0\}\cup\{s\in\mathbb Z:\delta\le|s|<n\},
\qquad
T_t:=\{r\in\mathbb Z:1\le|r|\le t\}.
\]

\begin{proposition}\label{prop:matrix-difference-code-design}
Let \(\chi=(\chi_{ab})\in A_+\setminus\{0\}\).
It is a \(\Delta_\delta\)-code if and only if
\[
\chi_{aa}\chi_{bb}=0
\qquad(a,b\in[n],\ 0<|a-b|<\delta).
\]
It is a \(T_t\)-design if and only if
\[
\chi_{ab}=0
\qquad(a,b\in[n],\ 1\le|a-b|\le t).
\]
\end{proposition}

\begin{proof}
For \(0<|a-b|<\delta\), Proposition~\ref{prop:matrix-difference-distributions} gives
\[
A_{a-b}^\chi(\mathsf F_{aa}^{(a-b)})
=\frac{\chi_{aa}\chi_{bb}}{\operatorname{Tr}(\chi^2)}.
\]
Thus the code condition implies the stated diagonal-entry condition.
For the converse, positivity gives \(|\chi_{ij}|^2\le\chi_{ii}\chi_{jj}\), so a zero diagonal entry forces its row and column to vanish.
If \(0<|s|<\delta\) and \(i,j\in D_s\), the condition \(\chi_{ii}\chi_{i-s,i-s}=0\) therefore implies \(\overline{\chi_{ij}}\chi_{i-s,j-s}=0\).
Hence \(A_s^\chi\) vanishes on every matrix unit, proving the code condition.
For designs, the values \(B_r^\chi(\mathsf E_{ii}^{(r)})\) are positive multiples of \(|\chi_{i,i-r}|^2\); their vanishing forces every other value on the matrix units to vanish.
Thus
\[
B_r^\chi=0
\quad\Longleftrightarrow\quad
\chi_{i,i-r}=0\quad(i\in D_r).
\qedhere
\]
\end{proof}

The conic bounds use the following positivity conditions.

\begin{lemma}\label{lem:matrix-difference-cones}
For integers \(r,s\) with \(|r|,|s|<n\),
\[
\begin{gathered}
K_r^+
=\left\{\sum_{i,j\in D_r}b_{ij}\mathsf E_{ij}^{(r)}
:(b_{ij})_{i,j\in D_r}\succeq0\right\},\\
\left\{\sum_{i,j\in D_s}c_{ij}\mathsf F_{ij}^{(s)}
:(c_{ij})_{i,j\in D_s}\succeq0\right\}
\subseteq R_s^+.
\end{gathered}
\]
\end{lemma}

\begin{proof}
The \(\mathsf E_{ij}^{(r)}\) act as matrix units on \(\operatorname{span}_{\C}\{\varepsilon_{i,i-r}:i\in D_r\}\) and vanish on its orthogonal complement, proving the formula for \(K_r^+\).
Let \(M=(m_{ij})_{i,j\in D_s}\succeq0\) and \(x:=\sum_{i,j\in D_s}m_{ij}\mathsf F_{ij}^{(s)}\).
Then \(x^\ast=x\).
To verify the quadratic inequality defining \(R_s^+\), write \(\chi_{ab}=\langle v_a,v_b\rangle\) with \(v_a\in\C^n\).
The vectors \(\overline{v_i}\otimes v_{i-s}\) have Gram matrix
\[
Z:=\bigl(\overline{\chi_{ij}}\chi_{i-s,j-s}\bigr)_{i,j\in D_s},
\]
so \(Z\succeq0\).
Since \(\overline M\succeq0\),
\[
\langle\chi,x(\chi)\rangle
=n\operatorname{Tr}(\overline MZ)\ge0.
\]
Thus \(x\in R_s^+\).
\end{proof}

\begin{proposition}\label{prop:matrix-difference-bounds}
For \(1\le\delta\le n\) and \(0\le t\le n-1\), the conic bounds are
\begin{gather}
\DelsCode(W_n,\Delta_\delta)
=n\left\lceil\frac n\delta\right\rceil,\label{eq:wn-code-bound}\\*
\DelsDes(W_n,T_t)
=\frac{n}{\left\lceil n/(t+1)\right\rceil}.
\label{eq:wn-design-bound}
\end{gather}
Every \(\Delta_\delta\)-code projection \(p\) satisfies
\[
\operatorname{rank}(p)\le\left\lceil\frac n\delta\right\rceil,
\]
and equality is attained.
The minimum effective size of a \(T_t\)-design is
\[
\min\{\esize(\chi):\chi\in A_+\setminus\{0\},\
\chi\text{ is a }T_t\text{-design}\}=n.
\]
\end{proposition}

\begin{proof}
We first derive a matrix estimate common to both conic programs, then construct feasible families attaining the bounds.
For a feasible pair \((\mathsf A,\mathsf B)\), with \(\mathsf B=Q'\mathsf A\), put
\[
X_s:=\bigl(\mathsf A_s(\mathsf F_{ji}^{(s)})\bigr)_{i,j\in D_s}
\qquad(|s|<n).
\]
The formula for \(Q\) gives
\begin{equation}\label{eq:wn-transition}
(Q'\mathsf A)_r(\mathsf E_{ij}^{(r)})
=n^2(X_{i-j})_{i-r,i}
\qquad(i,j\in D_r).
\end{equation}
The involution condition makes \(X_s\) Hermitian.
Since
\[
\mathsf A_s\!\left(\sum_{i,j\in D_s}c_{ij}\mathsf F_{ij}^{(s)}\right)
=\operatorname{Tr}(X_sC)
\qquad\bigl(C=(c_{ij})_{i,j\in D_s}\bigr),
\]
Lemma~\ref{lem:matrix-difference-cones} and self-duality of the positive semidefinite cone give \(X_s\succeq0\)~\cite[Section~2.4]{BoydVandenberghe2004}.
The MacWilliams relation, normalization, and positivity of \(\mathsf B\) also give
\[
(X_0)_{ij}
=n^{-2}\mathsf B_{j-i}(\mathsf E_{jj}^{(j-i)})\ge0,
\qquad
\sum_{i,j\in[n]}(X_0)_{ij}
=\mathsf A_0(I)=1,
\]
and
\[
\mathsf B_0(E_0)=\mathsf A(J)
=n\sum_{|s|<n}\operatorname{Tr}(X_s).
\]

The estimate needed below follows by grouping nearby indices.
For an integer \(1\le d\le n\), every matrix \(M=(M_{ab})_{a,b\in[n]}\succeq0\) with \(M_{ab}=0\) for \(0<|a-b|<d\) satisfies
\begin{equation}\label{eq:wn-interval-estimate}
\sum_{a,b\in[n]}M_{ab}
\le\left\lceil\frac nd\right\rceil\operatorname{Tr}(M).
\end{equation}
Write \(M_{ab}=\langle u_a,u_b\rangle\) and partition \([n]\) into consecutive intervals \(I_1,\ldots,I_{\lceil n/d\rceil}\), each of size at most \(d\).
The vanishing entries make the vectors within each interval pairwise orthogonal, so Cauchy--Schwarz gives
\[
\left\|\sum_{a\in[n]}u_a\right\|^2
\le\left\lceil\frac nd\right\rceil
\sum_{j=1}^{\lceil n/d\rceil}
\left\|\sum_{a\in I_j}u_a\right\|^2
=\left\lceil\frac nd\right\rceil
\sum_{a\in[n]}\|u_a\|^2.
\]

For the code bound, apply this estimate to the matrix representing \(n^{-2}\mathsf B_0\):
\[
M:=\bigl((X_{b-a})_{bb}\bigr)_{a,b\in[n]},
\qquad M_{ab}=n^{-2}\mathsf B_0(\mathsf E_{ba}^{(0)}).
\]
Positivity on \(K_0^+\) gives \(M\succeq0\), and the code constraints give \(M_{ab}=0\) for \(0<|a-b|<\delta\).
The entries of \(X_0\) are nonnegative and sum to \(1\), so
\[
\operatorname{Tr}(M)=\operatorname{Tr}(X_0)\le1,
\qquad
\frac{\mathsf A(J)}n
=\sum_{|s|<n}\operatorname{Tr}(X_s)
=\sum_{a,b\in[n]}M_{ab}.
\]
Set \(m:=\lceil n/\delta\rceil\).
By \eqref{eq:wn-interval-estimate} with \(d=\delta\),
\[
\frac{\mathsf A(J)}n
\le m\operatorname{Tr}(M)\le m.
\]
The projection
\[
p:=\sum_{j=0}^{m-1}\varepsilon_{1+j\delta,1+j\delta}
\]
is a \(\Delta_\delta\)-code of rank \(m\), by Proposition~\ref{prop:matrix-difference-code-design}.
Since \(\esize(p)=n\operatorname{rank}(p)=nm\), it attains the code bound and also proves that the rank bound is attained.
For any code projection, \(\esize(p)=n\operatorname{rank}(p)\) gives the stated rank inequality.

For a feasible design family, the design constraints and \eqref{eq:wn-transition} give
\[
(X_0)_{ab}=0\qquad(1\le|a-b|\le t).
\]
Set \(m:=\lceil n/(t+1)\rceil\).
Applying \eqref{eq:wn-interval-estimate} to \(X_0\), with \(d=t+1\), gives
\[
1=\sum_{a,b\in[n]}(X_0)_{ab}
\le m\operatorname{Tr}(X_0).
\]
Since every \(X_s\) is positive semidefinite,
\[
\mathsf B_0(E_0)
=n\sum_{|s|<n}\operatorname{Tr}(X_s)
\ge n\operatorname{Tr}(X_0)\ge\frac nm.
\]

To attain \(n/m\), choose indices at least \(t+1\) apart and put
\[
S:=\{1+j(t+1):j=0,\ldots,m-1\}\subseteq[n],
\qquad
c:=\sum_{i,j\in S}\varepsilon_{ij},
\]
and define
\[
\mathsf A_s:=0\quad(s\ne0),
\qquad
\mathsf A_0(x):=\frac{\langle c,x(c)\rangle}{nm^2}
\quad(x\in V_0).
\]
Since \(c\in A_+\) and \(\operatorname{Tr}(c^2)=m^2\), the defining properties of \(R_0^+\) and the involution identity for distributions give positivity, Hermitian symmetry, and \(\mathsf A_0(I)=1\).
For \(\mathsf B:=Q'\mathsf A\), \eqref{eq:wn-transition} gives
\[
\mathsf B_r(\mathsf E_{ij}^{(r)})=
\begin{cases}
n^2m^{-2},&i=j,\ i\in S,\ i-r\in S,\\
0,&\text{otherwise.}
\end{cases}
\]
The matrices representing these functionals are diagonal and positive semidefinite, so every \(\mathsf B_r\) is Hermitian and nonnegative on \(K_r^+\).
The separation of the indices in \(S\) gives \(\mathsf B_r=0\) for \(1\le|r|\le t\).
Thus the pair is feasible, with objective
\[
\mathsf B_0(E_0)
=\frac1n\sum_{i\in[n]}\mathsf B_0(\mathsf E_{ii}^{(0)})
=\frac nm.
\]
This proves \eqref{eq:wn-design-bound}.

To determine the minimum for actual designs, the nonnegative eigenvalues of every \(\chi\in A_+\setminus\{0\}\) give
\[
\esize(\chi)
=n\frac{\operatorname{Tr}(\chi)^2}{\operatorname{Tr}(\chi^2)}
\ge n.
\]
Equality holds for the \(T_t\)-design \(\varepsilon_{11}\), by Proposition~\ref{prop:matrix-difference-code-design}.
\end{proof}

When \(\lceil n/(t+1)\rceil>1\), the conic design optimum is strictly smaller than the minimum effective size \(n\) of a nonzero positive element of \(M_n(\C)\).
Thus some feasible pairs of functionals in the design program do not arise as distributions of positive elements.
The proof uses the inclusion in Lemma~\ref{lem:matrix-difference-cones}; it does not determine all \(R_s^+\) or provide semidefinite representations of the complete programs.
\section{Quantum Schurian schemes}\label{sec:schurian}

The classical Schurian construction takes the commutant of a transitive permutation action~\cite{Higman1970}.
We replace scalar-valued functions on cosets with matrix-valued functions determined by a representation of the stabilizer.

Let a finite group \(G\) act on a finite quantum space \(A\) by unital \(\ast\)-algebra automorphisms preserving \(h_A\).
The action is unitary for \(\langle x,y\rangle=h_A(x^\ast y)\).
Define
\[
\End_G(A):=
\{T\in\End_{\C}(A):
T(g\cdot x)=g\cdot T(x)\quad(g\in G,\ x\in A)\}.
\]

\begin{lemma}\label{lem:equivariant-endomorphism-algebra}
The space \(\End_G(A)\) is a quantum coherent algebra.
\end{lemma}

\begin{proof}
Let \(S,T\in\End_G(A)\).
Preservation of \(1_A,h_A,\ast\) and unitarity give
\[
I,J,ST,S^\dagger,S^\ast\in\End_G(A).
\]
Preservation of multiplication makes \(m\) equivariant for the diagonal action on \(A\otimes A\).
Unitarity then makes \(m^\dagger\) equivariant, so
\[
S\mathbin{\bullet}T=m(S\otimes T)m^\dagger\in\End_G(A).
\]
\end{proof}

Let \(H\le G\), and let \(V\ne\{0\}\) be a finite-dimensional Hilbert space with unitary group \(U(V)\).
For a unitary representation \(\tau:H\to U(V)\), put \(d=\dim V\) and
\[
A:=A_{G,H,\tau}
 :=\left\{f:G\to\End(V):
 \begin{array}{l}
 f(gh)=\tau(h)^{-1}f(g)\tau(h)\\
 (g\in G,\ h\in H)
 \end{array}\right\}.
\]
For \(f,f'\in A\), the operations, unit, and trace are
\[
\begin{aligned}
(ff')(g)&=f(g)f'(g),
&f^\ast(g)&=f(g)^\ast,\\
1_A(g)&=I_V,
&h_A(f)&=d\sum_{gH\in G/H}\operatorname{Tr}(f(g)).
\end{aligned}
\]
Since \(\operatorname{Tr}(f(gh))=\operatorname{Tr}(f(g))\), the trace is independent of the coset representatives.
Evaluation at one representative of each coset gives
\[
A\cong\bigoplus_{gH\in G/H}\End(V),
\qquad h_A(1_A)=|G/H|d^2.
\]
Define left translation and its commutant by
\[
(s\cdot f)(g):=f(s^{-1}g)
\quad(s,g\in G),
\qquad W:=\End_G(A).
\]

\begin{theorem}\label{thm:quantum-schurian-construction}
Left translation acts on the finite quantum space \(A\), and \(W\) is a QAS.
\end{theorem}

\begin{proof}
Left translation preserves the pointwise operations, and
\[
\begin{aligned}
(s\cdot f)(gh)
&=\tau(h)^{-1}(s\cdot f)(g)\tau(h),\\
h_A(s\cdot f)
&=d\sum_{gH\in G/H}\operatorname{Tr}(f(s^{-1}g))
 =h_A(f).
\end{aligned}
\]
Thus Lemma~\ref{lem:equivariant-endomorphism-algebra} applies.
To prove minimality of \(I\), let \(p\in W\) satisfy
\[
p\mathbin{\bullet}p=p=p^\ast,
\qquad p\mathbin{\bullet}I=I\mathbin{\bullet}p=p.
\]
By~\cite[Lemma~2.11]{Gromada2024}, there is a subset \(C\subseteq G/H\) such that
\[
(pf)(g)=
\begin{cases}
f(g),&gH\in C,\\
0,&gH\notin C.
\end{cases}
\]
Equivariance gives \(sC=C\) for every \(s\in G\).
Since \(G\) acts transitively on \(G/H\),
\[
C\in\{\varnothing,G/H\}
\quad\Longrightarrow\quad p\in\{0,I\}.
\]
\end{proof}

\begin{definition}\label{def:quantum-schurian}
The QAS \(\End_G(A_{G,H,\tau})\) is the quantum Schurian scheme associated with \(G,H,\tau\).
For \(v\ge2\), \(1\le k\le v-1\), and a unitary representation \(\tau\) of
\[
H=\{g\in S_v:g([k])=[k]\}\cong S_k\times S_{v-k},
\]
the quantum Johnson scheme is
\[
\mathfrak J_\tau(v,k):=\End_{S_v}(A_{S_v,H,\tau}).
\]
\end{definition}

For \(V=\C\) and \(\tau\equiv1\),
\[
A\cong\C^{G/H},
\qquad W=\End_G(\C^{G/H}),
\]
so \(W\) is the classical Schurian adjacency algebra~\cite{Higman1970}.
We apply the construction to four families below.

\begin{example}\label{ex:quantum-johnson}
For the Johnson scheme, let \(v\ge2\) and \(1\le k\le v-1\).
Take
\[
\begin{aligned}
G&=S_v,\\
H&=\{g\in S_v:g([k])=[k]\}
   \cong S_k\times S_{v-k}.
\end{aligned}
\]
Then
\[
\begin{aligned}
G/H&\xrightarrow{\sim}\{B\subseteq[v]:|B|=k\},\\
gH&\longmapsto g([k]).
\end{aligned}
\]
For a unitary representation \(\tau:H\to U(V)\), Theorem~\ref{thm:quantum-schurian-construction} gives the QAS
\[
\mathfrak J_\tau(v,k)=\End_G(A_{G,H,\tau}).
\]
Under this identification, the orbits on ordered pairs are
\[
\{(B,C):|B\cap C|=k-i\},
\qquad i=0,\ldots,\min(k,v-k).
\]
Pairs with the same \(|B\cap C|\) have equally sized corresponding parts
\[
B\cap C,\quad B\setminus C,\quad C\setminus B,
\quad[v]\setminus(B\cup C).
\]
Bijections between these parts give the required permutation.
For \(V=\C\) and \(\tau\equiv1\), this recovers \(J(v,k)\)~\cite[Examples~2.1 and~3.1]{MartinTanaka2009}.
\end{example}

\begin{example}\label{ex:quantum-grassmann}
For the Grassmann scheme, let \(q\) be a prime power, \(v\ge2\), and \(1\le k\le v-1\).
Relative to \(\mathbb F_q^v=\mathbb F_q^k\oplus\mathbb F_q^{v-k}\), take
\[
\begin{aligned}
G&=\operatorname{GL}_v(\mathbb F_q),\\
H&=\{g\in G:g(\mathbb F_q^k\oplus0)=\mathbb F_q^k\oplus0\}\\
 &=\left\{
 \begin{pmatrix}P&Q\\0&R\end{pmatrix}:
 \begin{array}{l}
 P\in\operatorname{GL}_k(\mathbb F_q),\\
 Q\in\mathbb F_q^{k\times(v-k)},\\
 R\in\operatorname{GL}_{v-k}(\mathbb F_q)
 \end{array}\right\}.
\end{aligned}
\]
Then
\[
\begin{aligned}
G/H&\xrightarrow{\sim}
 \{B\le\mathbb F_q^v:\dim B=k\},\\
gH&\longmapsto g(\mathbb F_q^k\oplus0).
\end{aligned}
\]
For a unitary representation \(\tau:H\to U(V)\), Theorem~\ref{thm:quantum-schurian-construction} gives the QAS \(\End_G(A_{G,H,\tau})\).
Under this identification, the orbits on ordered pairs are
\[
\{(B,C):\dim(B\cap C)=k-i\},
\qquad i=0,\ldots,\min(k,v-k).
\]
Extend a basis of \(B\cap C\) to bases of \(B\) and \(C\), then extend their union to a basis of \(\mathbb F_q^v\).
For two pairs with the same intersection dimension, matching these bases gives the required invertible linear map.
For \(V=\C\) and \(\tau\equiv1\), this recovers \(J_q(v,k)\)~\cite[Section~2]{Kurihara2011}.
\end{example}

\begin{example}\label{ex:quantum-hamming-schurian}
For the Schurian construction of the Hamming scheme, let \(n\ge1\) and \(q\ge2\) be integers.
Take
\[
\begin{aligned}
G&=S_q\wr S_n=S_q^n\rtimes S_n,\\
H&=S_{q-1}\wr S_n=S_{q-1}^n\rtimes S_n,
\end{aligned}
\]
where \(S_{q-1}=\{\pi\in S_q:\pi(1)=1\}\), and \(S_n\) permutes the factors.
The action and coset identification are
\[
\bigl((\pi_1,\ldots,\pi_n),\sigma\bigr)\cdot x
 =\bigl(\pi_i(x_{\sigma^{-1}(i)})\bigr)_{i=1}^n,
\]
\[
\begin{aligned}
G/H&\xrightarrow{\sim}[q]^n,\\
gH&\longmapsto g\cdot(1,\ldots,1).
\end{aligned}
\]
For a unitary representation \(\tau:H\to U(V)\), Theorem~\ref{thm:quantum-schurian-construction} gives the QAS \(\End_G(A_{G,H,\tau})\).
Under this identification, the orbits on ordered pairs are
\[
\{(x,y):|\{j\in[n]:x_j\ne y_j\}|=i\},
\qquad i=0,\ldots,n.
\]
Coordinate permutations match the positions where the words differ; symbol permutations then match each ordered pair of symbols.
For \(V=\C\) and \(\tau\equiv1\), this recovers \(H(n,q)\)~\cite[Examples~2.2 and~3.2]{MartinTanaka2009}.
\end{example}

\begin{remark}
For a prime power \(q\) and \(n\ge1\), the Schurian Hamming construction with alphabet size \(q^2\) and the twisted Hamming construction in Section~\ref{sec:translation-hamming} quantize \(H(n,q^2)\) on the respective algebras
\[
\bigoplus_{x\in[q^2]^n}\End(V)
\qquad\text{and}\qquad M_{q^n}(\C).
\]
The first has \(q^{2n}>1\) simple summands, and the second is simple, so these are different finite quantum spaces.
\end{remark}

\begin{example}\label{ex:quantum-bilinear-forms}
For the bilinear forms scheme, let \(q\) be a prime power, \(r,s\ge1\), and put \(\Omega=\mathbb F_q^{r\times s}\), with its additive group structure.
The linear action used in the semidirect product is
\[
(L,R)\cdot X:=LXR^{-1}
\qquad
\left(
\begin{array}{l}
L\in\operatorname{GL}_r(\mathbb F_q),\
R\in\operatorname{GL}_s(\mathbb F_q),\\
X\in\Omega
\end{array}
\right).
\]
Take
\[
\begin{aligned}
G&=\Omega\rtimes
 \bigl(\operatorname{GL}_r(\mathbb F_q)
 \times\operatorname{GL}_s(\mathbb F_q)\bigr),\\
H&=\{(0,L,R):L\in\operatorname{GL}_r(\mathbb F_q),\
 R\in\operatorname{GL}_s(\mathbb F_q)\}.
\end{aligned}
\]
The action and coset identification are
\[
(B,L,R)\cdot X=B+LXR^{-1},
\]
\[
\begin{aligned}
G/H&\xrightarrow{\sim}\Omega,\\
(B,L,R)H&\longmapsto B.
\end{aligned}
\]
For a unitary representation \(\tau:H\to U(V)\), Theorem~\ref{thm:quantum-schurian-construction} gives the QAS \(\End_G(A_{G,H,\tau})\).
Under this identification, the orbits on ordered pairs are
\[
\{(X,Y):\operatorname{rank}(X-Y)=i\},
\qquad i=0,\ldots,\min(r,s).
\]
Translation sends \((X,Y)\) to \((X-Y,0)\); matrices of the same rank are related by invertible row and column operations.
For \(V=\C\) and \(\tau\equiv1\), this recovers the classical bilinear forms scheme~\cite[Section~2]{Kurihara2011}.
\end{example}

\begin{prob}
Quantize the distance schemes of twisted Grassmann graphs of diameter at least two and the two-class schemes of Chang graphs on noncommutative finite quantum spaces, and explain how the classical schemes are recovered.
\end{prob}

Twisted Grassmann graphs of diameter at least two~\cite{VanDamKoolen2005} and Chang graphs~\cite[Proposition~4.3 and its proof]{AbdollahiVanDamJazaeri2017} are not vertex-transitive.
Since translation and Schurian schemes admit transitive groups of automorphisms, these schemes are neither translation nor Schurian.
Thus the constructions above do not apply directly.
\section{Quantum block designs}\label{sec:quantum-block-designs}

The subgroup-averaging method below applies to all four constructions in Section~\ref{sec:schurian}, with an appropriate choice of subgroups.
Here we treat only block designs: averaging over permutations fixing prescribed points gives a design condition in the quantum Johnson scheme and recovers the classical incidence condition for scalar weights.

Let \(v\ge2\) and \(1\le k\le v-1\), and put
\[
\mathcal B:=\binom{[v]}k,
\qquad
H:=\{\sigma\in S_v:\sigma([k])=[k]\}\cong S_k\times S_{v-k}.
\]
\begin{definition}\label{def:weighted-block-design}
For \(0\le t\le k\), a family \(w\in\mathbb R_{\ge0}^{\mathcal B}\setminus\{0\}\) is a weighted \(t\)-\((v,k)\) design if
\[
\sum_{\substack{B\in\mathcal B\\B\supseteq C}}w_B
=\frac{\binom{k}{t}}{\binom{v}{t}}\sum_{B\in\mathcal B}w_B
\qquad\left(C\in\binom{[v]}t\right).
\]
\end{definition}
Equivalently, the left-hand side is independent of \(C\); summing over all \(t\)-subsets determines its common value~\cite{Delsarte1973,MartinTanaka2009}.

Throughout this section, let \(\tau:H\to U(V)\) be a nonzero irreducible finite-dimensional unitary representation, put \(d=\dim V\), and set \(A=A_{S_v,H,\tau}\) as in Example~\ref{ex:quantum-johnson}.
The induced representation decomposes into mutually orthogonal subspaces indexed by the blocks:
\[
\operatorname{Ind}_H^{S_v}(\tau)\cong\bigoplus_{B\in\mathcal B}V_B,
\qquad
\dim V_B=d,
\qquad
gV_B=V_{gB}
\qquad(g\in S_v)
\]
by the induced representation construction~\cite{Serre1977}.
This identifies the algebra and trace with
\[
A\cong\bigoplus_{B\in\mathcal B}\End(V_B),
\qquad
h_A((x_B)_B)=d\sum_{B\in\mathcal B}\operatorname{Tr}(x_B).
\]
The group action transports the components by unitary conjugation.
For \(C\subseteq[v]\), put
\[
K_C:=\{g\in S_v:g(c)=c\quad(c\in C)\}.
\]
For an \(S_v\)-module \(U\) and \(K\le S_v\), write
\[
U^K:=\{u\in U:g\cdot u=u\quad(g\in K)\}.
\]

The average over \(K_C\) is the orthogonal projection onto \(A^{K_C}\).
We require it to retain only the constant component.
\begin{definition}\label{def:quantum-block-tau-design}
For \(0\le t\le k\), an element \(\chi\in A_+\setminus\{0\}\) is a quantum block \(\tau\)-design of strength \(t\) if
\begin{equation}\label{eq:tau-quantum-block-design}
\frac1{|K_C|}\sum_{g\in K_C}g\cdot\chi=\frac{h_A(\chi)}{h_A(1_A)}1_A
\qquad\left(C\in\binom{[v]}t\right).
\end{equation}
\end{definition}

To express this averaging condition in terms of the distributions from Section~\ref{sec:delsarte-theory}, we determine which irreducible components have \(K_C\)-fixed vectors.
For each partition \(\lambda\vdash v\), let \(S^\lambda\) be the irreducible complex representation of \(S_v\) indexed by \(\lambda\), and let \(\lambda_1\) be its largest part~\cite[Chapter~2]{Sagan2001}.
Let \(e_\lambda\) be the orthogonal projection onto the sum of the submodules of \(A\) isomorphic to \(S^\lambda\), and set \(e_\lambda=0\) if none occur.
Complete reducibility and Schur's lemma identify the nonzero \(e_\lambda\) as the minimal central projections of \(\mathfrak J_\tau(v,k)\) for composition~\cite{Serre1977}.
An \(S_v\)-invariant element is a constant matrix commuting with \(\tau(H)\).
Irreducibility of \(\tau\) and Schur's lemma therefore give
\[
A^{S_v}=\C1_A,
\qquad
e_{(v)}=E_0.
\]
Put
\[
T_t:=\{\lambda\vdash v:e_\lambda\ne0,\ 1\le v-\lambda_1\le t\}.
\]

\begin{theorem}\label{thm:johnson-pointwise-characterization}
Let \(0\le t\le k\) and \(\chi\in A_+\setminus\{0\}\).
Then \(\chi\) is a quantum block \(\tau\)-design of strength \(t\) if and only if it is a \(T_t\)-design in \(\mathfrak J_\tau(v,k)\).
Equivalently,
\[
e_\lambda\chi=0
\qquad(\lambda\vdash v,\ 1\le v-\lambda_1\le t).
\]
For the one-dimensional trivial representation \(\tau\), these conditions are precisely those of a weighted \(t\)-\((v,k)\) design.
\end{theorem}

\begin{proof}
First we identify the components excluded by the averaging condition.
For each \(C\in\binom{[v]}t\), the branching rule for \(K_C\cong S_{v-t}\) gives
\[
(S^\lambda)^{K_C}\ne0
\quad\Longleftrightarrow\quad
\lambda_1\ge v-t
\]
for every partition \(\lambda\vdash v\)~\cite[Section~2.8]{Sagan2001}.
Indeed, the trivial representation indexed by \((v-t)\) occurs on restriction exactly when this one-row partition is contained in \(\lambda\), which is equivalent to \(\lambda_1\ge v-t\).
For an irreducible \(S_v\)-submodule \(U\subseteq A\), the translates of \(U^{K_C}\) span an invariant subspace.
By irreducibility, this span is \(U\) when \(U^{K_C}\ne0\), and is zero otherwise.
Conjugation gives \(gK_Cg^{-1}=K_{gC}\), so applying this observation to each irreducible summand yields
\[
\operatorname{span}_{\C}\bigcup_{C\in\binom{[v]}t} A^{K_C}
=\C1_A\oplus\bigoplus_{\lambda\in T_t}\operatorname{im}e_\lambda.
\]
Since
\[
E_0\chi=\frac{h_A(\chi)}{h_A(1_A)}1_A,
\qquad \chi-E_0\chi\perp\C1_A,
\]
the decomposition shows that \(\chi-E_0\chi\perp A^{K_C}\) for every \(C\) exactly when \(e_\lambda\chi=0\) for every \(\lambda\in T_t\).
Averaging over \(K_C\) projects onto \(A^{K_C}\), so this orthogonality is precisely \eqref{eq:tau-quantum-block-design}.
Lemma~\ref{lem:spectral-vanishing} then identifies it with the \(T_t\)-design condition.

It remains to compare the averaging condition with the classical incidence condition.
For trivial \(\tau\), write \(\chi=(w_B)_{B\in\mathcal B}\) and put
\[
c:=\frac1{|\mathcal B|}\sum_{B\in\mathcal B}w_B.
\]
For fixed \(C\), the nonempty sets
\[
\{B\in\mathcal B:B\cap C=S\},\qquad S\subseteq C,
\]
are the \(K_C\)-orbits: permutations of \([v]\setminus C\) act transitively on the remaining elements of a block.
Thus \eqref{eq:tau-quantum-block-design} requires mean weight \(c\) on each nonempty orbit.
The blocks containing \(C\) form one such orbit, so their total weight is
\[
c\binom{v-t}{k-t}
=\frac{\binom{k}{t}}{\binom{v}{t}}\sum_{B\in\mathcal B}w_B,
\]
as required.

Conversely, suppose that \(w\) is a weighted \(t\)-\((v,k)\) design.
The incidence condition gives \(\sum_{B\supseteq C}(w_B-c)=0\) whenever \(|C|=t\).
We first obtain the same vanishing for smaller prescribed subsets by counting their \(t\)-subset extensions.
For \(S\subseteq[v]\) with \(|S|\le t\),
\[
\binom{k-|S|}{t-|S|}\sum_{B\supseteq S}(w_B-c)
=\sum_{\substack{C\supseteq S\\|C|=t}}\sum_{B\supseteq C}(w_B-c)=0.
\]
Since \(t\le k\), the coefficient is positive, giving \(\sum_{B\supseteq S}(w_B-c)=0\).
We can now isolate blocks having a specified intersection with \(C\).
For \(|C|=t\) and \(S\subseteq C\), inclusion--exclusion gives
\[
\sum_{B:B\cap C=S}(w_B-c)
=\sum_{S\subseteq R\subseteq C}(-1)^{|R|-|S|}\sum_{B\supseteq R}(w_B-c)=0.
\]
Every nonempty orbit therefore has mean weight \(c\), proving the averaging condition.
\end{proof}

Component traces and scalar matrices relate the two design conditions as follows.

\begin{samepage}
\begin{proposition}\label{prop:johnson-fiber-traces}
Let \(0\le t\le k\), and let \(\chi=(\chi_B)_{B\in\mathcal B}\in A_+\setminus\{0\}\) be a \(T_t\)-design.
Then \((\operatorname{Tr}(\chi_B))_{B\in\mathcal B}\) is a weighted \(t\)-\((v,k)\) design.
Conversely, for every weighted \(t\)-\((v,k)\) design \(w\), the element
\[
(w_BI_{V_B})_{B\in\mathcal B}
\]
is a \(T_t\)-design.
\end{proposition}
\end{samepage}

\begin{proof}
Write \(w_B=\operatorname{Tr}(\chi_B)\).
These weights are nonnegative and not all zero, since \(\chi\) is positive and nonzero.
By Theorem~\ref{thm:johnson-pointwise-characterization}, \(\chi\) satisfies \eqref{eq:tau-quantum-block-design}.
The action transports components by unitary conjugation, so taking their traces gives
\[
\frac1{|K_C|}\sum_{g\in K_C}w_{g^{-1}B}=
\frac1{|\mathcal B|}\sum_{D\in\mathcal B}w_D
\qquad\left(B\in\mathcal B,\ C\in\binom{[v]}t\right).
\]
The scalar case of Theorem~\ref{thm:johnson-pointwise-characterization} now gives the weighted classical design.
Conversely, for a weighted classical design, that theorem gives the scalar averaging condition.
Since unitary conjugation sends identity matrices to identity matrices, the same condition gives \eqref{eq:tau-quantum-block-design} for \((w_BI_{V_B})_B\).
Applying the theorem for the given \(\tau\) proves the converse.
\end{proof}

\begin{remark}\label{rem:johnson-zauner-orbits}
For projections, the averaging condition in Definition~\ref{def:quantum-block-tau-design} can be expressed through Zauner's coherence condition.
Fix \(0\le t\le k\), and let
\[
\mathcal H:=\bigoplus_{B\in\mathcal B}V_B,
\qquad
\rho:S_v\longrightarrow U(\mathcal H)
\]
be the induced representation, and identify \(A\) with the block diagonal subalgebra of \(\End(\mathcal H)\).
For \(P\in A\) satisfying
\[
P^2=P=P^\ast,
\qquad
0\ne P\ne1_A=I_{\mathcal H},
\]
put
\[
D_C(P):=\{\rho(g)P\rho(g)^\ast:g\in K_C\}
\qquad\left(C\in\binom{[v]}t\right).
\]
Each orbit element occurs equally often in the group average, so
\[
\frac1{|D_C(P)|}\sum_{Q\in D_C(P)}Q
=\frac1{|K_C|}\sum_{g\in K_C}\rho(g)P\rho(g)^\ast.
\]
Since \(h_A=d\operatorname{Tr}_{\mathcal H}\), we also have
\[
\frac{h_A(P)}{h_A(1_A)}
=\frac{\operatorname{rank}P}{\dim\mathcal H}.
\]
Consequently,
\begin{equation}\label{eq:johnson-zauner-coherence}
\begin{aligned}
&P\text{ is a quantum block }\tau\text{-design of strength }t\\
&\quad\Longleftrightarrow\quad
\sum_{Q\in D_C(P)}Q
=\frac{|D_C(P)|\operatorname{rank}P}{\dim\mathcal H}I_{\mathcal H}
\quad\left(C\in\binom{[v]}t\right).
\end{aligned}
\end{equation}
The projections in each \(D_C(P)\) have the same rank, and the right-hand side is Zauner's coherence condition~\cite[Definition~1.1]{Zauner1999}.
Under this condition, each \(D_C(P)\) is regular and coherent in his sense: a singleton orbit would force \(P\) to be a scalar projection, contradicting \(0\ne P\ne I_{\mathcal H}\).
Here the strength \(t\) specifies the subgroups \(K_C\), while the coherence condition on each orbit concerns the first tensor moment in Zauner's definition.
\end{remark}

\begin{prob}\label{prob:johnson-zauner-reconstruction}
Fix \(2\le t\le k\le v-1\) and an irreducible unitary representation \(\tau:H\to U(V)\).
For a finite set \(D\subseteq\End(\mathcal H)\) of at least two projections that is regular and coherent in Zauner's sense, determine necessary and sufficient conditions for the existence of
\[
u\in U(\mathcal H),\qquad
C\in\binom{[v]}t,\qquad
P\in A,\qquad
P^2=P=P^\ast,\quad0\ne P\ne1_A,
\]
such that \(P\) is a quantum block \(\tau\)-design of strength \(t\) and
\[
uDu^\ast=D_C(P).
\]
Give a construction of such \(P\).
\end{prob}
\section{An optimal quantum block design}\label{sec:optimal-block-design}

We construct two complementary projections that are designs of strength \(2\) in a quantum Johnson scheme, lie outside the center of the underlying algebra, and attain the minimum effective size \(20\).
The conic bound is \(40/3\); testing only the minimal central projections for the quantum Schur product gives \(10\).

Identify \(S_5=\operatorname{Sym}(X)\), where \(X:=\mathbb F_5\), and put
\[
\mathcal B:=\binom X2,\qquad
H:=\{g\in S_5:g\{0,1\}=\{0,1\}\}\cong S_2\times S_3.
\]
For \(B\in\mathcal B\) and \(g\in S_5\), coordinate permutation gives unitary maps
\[
\begin{aligned}
V_B&:=\left\{z\in\C^{X\setminus B}:
             \sum_{x\notin B}z_x=0\right\},
&\dim V_B&=2,\\
g:V_B&\longrightarrow V_{gB},
&(gz)_{g(x)}&:=z_x.
\end{aligned}
\]
The restriction
\[
\tau:H\longrightarrow U(V_{\{0,1\}}),\qquad
\tau(h)z:=hz
\]
is trivial on \(S_2\) and standard on \(S_3\), hence irreducible.
Set
\[
\begin{aligned}
A&:=\bigoplus_{B\in\mathcal B}\End(V_B),
&h_A(x)&:=2\sum_{B\in\mathcal B}\operatorname{Tr}(x_B),\\
(g\cdot x)_{gB}&:=g x_Bg^{-1},
&W&:=\End_{S_5}(A).
\end{aligned}
\]
Thus \(h_A(1_A)=40\).
The map
\[
A_{S_5,H,\tau}\longrightarrow A,\qquad
f\longmapsto(gf(g)g^{-1})_{g\{0,1\}}
\]
is a trace-preserving equivariant \(\ast\)-isomorphism.
Its inverse sends \(x\in A\) to \(f(g)=g^{-1}x_{g\{0,1\}}g\).
The forward map is independent of coset representatives by the transformation rule for \(f(gh)\).
It identifies \(W\) with \(\mathfrak J_\tau(5,2)\), preserving both products.

\begin{proposition}\label{prop:johnson-module-decomposition}
The \(S_5\)-module and composition algebra decompose as
\begin{equation}\label{eq:johnson-module-decomposition}
\begin{aligned}
A&\cong S^{(5)}\oplus2S^{(4,1)}\oplus2S^{(3,2)}
       \oplus2S^{(3,1,1)}
       \oplus S^{(2,2,1)}\oplus S^{(2,1,1,1)},\\
(W,\circ,\dagger)&\cong\C^{\oplus3}\oplus M_2(\C)^{\oplus3}.
\end{aligned}
\end{equation}
In particular, composition is noncommutative and \(\dim W=15\).
\end{proposition}

\begin{proof}
Let \(\mathbf1\) denote the trivial representation of either group factor, let \(\operatorname{sgn}\) and \(\operatorname{Std}\) denote the sign and standard representations of \(S_3\), and let \(\boxtimes\) denote the external tensor product.
Conjugation on \(\End(V_{\{0,1\}})\) gives
\[
\begin{aligned}
\End(V_{\{0,1\}})
&\cong\mathbf1\boxtimes
(\mathbf1\oplus\operatorname{sgn}\oplus\operatorname{Std}),\\
A&\cong\operatorname{Ind}_H^{S_5}\End(V_{\{0,1\}}).
\end{aligned}
\]
The Littlewood--Richardson rule~\cite[Chapter~4]{Sagan2001} gives
\[
\begin{aligned}
\operatorname{Ind}_{S_2\times S_3}^{S_5}
(\mathbf1\boxtimes\mathbf1)
&\cong S^{(5)}\oplus S^{(4,1)}\oplus S^{(3,2)},\\
\operatorname{Ind}_{S_2\times S_3}^{S_5}
(\mathbf1\boxtimes\operatorname{sgn})
&\cong S^{(3,1,1)}\oplus S^{(2,1,1,1)},\\
\operatorname{Ind}_{S_2\times S_3}^{S_5}
(\mathbf1\boxtimes\operatorname{Std})
&\cong S^{(4,1)}\oplus S^{(3,2)}
       \oplus S^{(3,1,1)}\oplus S^{(2,2,1)}.
\end{aligned}
\]
Schur's lemma gives the composition algebra from the multiplicities~\cite[Chapter~2]{Serre1977}.
\end{proof}

For the isotypic projections from Section~\ref{sec:quantum-block-designs},
\[
e_{(5)}=E_0=J/40,\qquad
T_2=\{(4,1),(3,2),(3,1,1)\}.
\]
For \(\chi\in A_+\setminus\{0\}\), Theorem~\ref{thm:johnson-pointwise-characterization} identifies strength \(2\) with
\[
e_\lambda\chi=0\qquad(\lambda\in T_2).
\]
For such a \(T_2\)-design \(\chi\), Proposition~\ref{prop:johnson-fiber-traces} makes \((\operatorname{Tr}(\chi_B))_B\) a weighted \(2\)-\((5,2)\) design.
Each two-element subset is contained in just one block of \(\mathcal B\), namely itself, so all these traces are equal.
These trace conditions do not imply \(e_{(3,1,1)}\chi=0\), since \(\C^{\mathcal B}\) contains no copy of \(S^{(3,1,1)}\).
Affine symmetry will enforce all three vanishing conditions while allowing rank one in each component.

For \(B=\{a,b\}\), let \(e_x\) be the coordinate vector at \(x\in X\setminus B\), and define
\[
\begin{aligned}
m_B&:=\frac{a+b}{2}\in X\setminus B,
&u_B&:=\frac{\sum_{x\notin B}e_x-3e_{m_B}}{\sqrt6}
       \in V_B,\\
P_B^+&:=u_Bu_B^*,
&P_B^-&:=I_{V_B}-P_B^+,\qquad
P^\pm:=(P_B^\pm)_B.
\end{aligned}
\]
Since \(\|u_B\|=1\) and \(\dim V_B=2\), both \(P_B^\pm\) are rank-one projections.

\begin{lemma}\label{lem:agl-fixed-vectors}
For the affine group
\[
K:=\{x\mapsto ux+v:
      u\in\mathbb F_5^\times,\ v\in\mathbb F_5\}\le S_5,
\]
we have \((S^\lambda)^K=\{0\}\) for every \(\lambda\in T_2\).
\end{lemma}

\begin{proof}
The cycle types and character values on \(K\) are
\[
\begin{array}{c|rrrr}
\text{cycle type}&1^5&5&2^2 1&4 1\\ \hline
\text{number in }K&1&4&5&10\\ \hline
S^{(4,1)}&4&-1&0&0\\
S^{(3,2)}&5&0&1&-1\\
S^{(3,1,1)}&6&1&-2&0
\end{array}
\]
The character values follow from the permutation representations on points and two-element subsets, and from
\(S^{(3,1,1)}\cong\bigwedge^2S^{(4,1)}\)~\cite[Chapter~2]{Sagan2001}.
The character-average formula~\cite[Chapter~2]{Serre1977} gives
\[
\dim(S^\lambda)^K
=\frac1{20}\sum_{g\in K}\operatorname{Tr}(g|_{S^\lambda})
=0
\qquad(\lambda\in T_2).\qedhere
\]
\end{proof}

Recall the pointwise stabilizers
\[
K_C=\{g\in S_5:g(c)=c\ (c\in C)\}
\qquad(C\in\mathcal B).
\]

\begin{theorem}
\label{thm:johnson-minimum-size}
\label{prop:johnson-size-twenty-design}
The projections \(P^\pm\in A\) are quantum block \(\tau\)-designs of strength \(2\) and are not central in \(A\).
They satisfy
\[
\begin{gathered}
P^+P^-=0,\qquad P^++P^-=1_A,\\
\sum_{B\in\mathcal B}\operatorname{rank}(P_B^\pm)=10,
\qquad \esize(P^\pm)=20.
\end{gathered}
\]
They attain
\[
\min\{\esize(\chi):
      \chi\in A_+\setminus\{0\},\
      e_\lambda\chi=0\ (\lambda\in T_2)\}
=20
\]
and satisfy
\begin{equation}\label{eq:johnson-pointwise-average}
\frac16\sum_{g\in K_C}g\cdot P^\pm=\frac12 1_A
\qquad(C\in\mathcal B).
\end{equation}
For \(\chi\in A_+\setminus\{0\}\) central in \(A\),
\[
\chi\text{ is a }T_2\text{-design}
\quad\Longleftrightarrow\quad
\chi\in\R_{>0}1_A.
\]
Every such design has effective size \(40\).
\end{theorem}

\begin{proof}
First we prove the lower bound.
For a \(T_2\)-design \(\chi\in A_+\setminus\{0\}\), Proposition~\ref{prop:johnson-fiber-traces} gives
\(\operatorname{Tr}(\chi_B)=s>0\) for all \(B\).
Nonnegative eigenvalues then give
\[
\operatorname{Tr}(\chi_B^2)\le s^2,\qquad
h_A(\chi)=20s,\qquad
h_A(\chi^2)\le20s^2,\qquad
\esize(\chi)\ge20.
\]
To attain the bound, note that affine maps preserve midpoints, so \(g u_B=u_{gB}\) for \(g\in K\).
Thus \(P^+\) is \(K\)-fixed, and equivariance of \(e_\lambda\), together with Lemma~\ref{lem:agl-fixed-vectors}, gives
\[
e_\lambda P^+=0\qquad(\lambda\in T_2).
\]
The same vanishing holds for \(P^-=1_A-P^+\).
Both projections have rank one in each \(M_2(\C)\) component, hence are not central in \(A\) and satisfy
\[
h_A(P^\pm)=2\cdot10=20=\esize(P^\pm).
\]
Theorem~\ref{thm:johnson-pointwise-characterization} gives strength \(2\) and~\eqref{eq:johnson-pointwise-average}, since \(|K_C|=6\) and \(h_A(1_A)=40\).
If \(\chi\) is central in \(A\), its components are scalar matrices; their equal traces force \(\chi=c1_A\) with \(c>0\).
Conversely, every such element satisfies the averaging condition and has effective size \(40\).
\end{proof}

\subsection{Conic bounds}

We determine the simple ideals for the quantum Schur product and reduce the conic program to three variables.

For \(w\in\End(A)\), define its component maps by
\[
w_{BC}:\End(V_C)\longrightarrow\End(V_B),\qquad
w(x)_B=\sum_Cw_{BC}(x_C).
\]
Equip \(\operatorname{Hom}(V_C,V_B)\) with \(\langle a,b\rangle=\operatorname{Tr}(a^*b)\).
Define the following linear correspondence on maps of the displayed form:
\[
\begin{aligned}
\operatorname{Hom}_{\C}(\End(V_C),\End(V_B))
&\longrightarrow\End_{\C}(\operatorname{Hom}(V_C,V_B)),\\
w_{BC}&\longmapsto[w_{BC}],\\
\bigl[2a(\cdot)b^*\bigr](c)&:=a\operatorname{Tr}(b^*c)
\qquad(a,b,c\in\operatorname{Hom}(V_C,V_B)).
\end{aligned}
\]

\begin{proposition}\label{prop:johnson-schur-decomposition}
For the quantum Schur product,
\[
(W,\bullet,\ast)
\cong\C^{\oplus7}\oplus M_2(\C)^{\oplus2}.
\]
The pair orbits \(|B\cap C|=2,1,0\) contribute \(\C^{\oplus3}\), \(M_2(\C)^{\oplus2}\), and \(\C^{\oplus4}\), respectively.
Moreover,
\[
\{u^\ast\mathbin\bullet u:u\in V_a\}\subseteq R_a^+
\qquad(a\in\mathcal C).
\]
\end{proposition}

\begin{proof}
The quantum Schur product acts separately on each \((B,C)\)-component.
Taking \(a,b\) to be matrix units gives bases on both sides of the correspondence.
For matrix units \(E_{pq}\) in an orthonormal basis of \(V_C\), the trace normalization gives
\[
m^\dagger(E_{pq})
=\frac12\sum_{r=1}^2E_{pr}\otimes E_{rq},
\]
so, for \(a,b,c,d\in\operatorname{Hom}(V_C,V_B)\), direct multiplication gives
\[
\begin{aligned}
\bigl(2a(\cdot)b^*\bigr)\bullet
\bigl(2c(\cdot)d^*\bigr)
&=\operatorname{Tr}(b^*c)\,2a(\cdot)d^*,\\
\bigl(2a(\cdot)b^*\bigr)^\ast
&=2b(\cdot)a^*.
\end{aligned}
\]
Thus \(w_{BC}\mapsto[w_{BC}]\) is a \(\ast\)-isomorphism onto \(\End(\operatorname{Hom}(V_C,V_B))\).
The group transports \(a\) to \(gag^{-1}\); consequently, the algebra on a pair orbit is the commutant of its stabilizer on \(\operatorname{Hom}(V_C,V_B)\).
Indeed, a commuting operator extends equivariantly to the orbit, independently of the chosen transporter.

For \(B=C\), the stabilizer representation on \(\End(V_B)\) is \(\mathbf1\oplus\operatorname{sgn}\oplus\operatorname{Std}\).
For \(|B\cap C|=1\), the stabilizer interchanges the two points outside \(B\cup C\); its eigenvalues \(1,-1\) each occur twice on \(\operatorname{Hom}(V_C,V_B)\).
For disjoint \(B,C\), the stabilizer \(S_2\times S_2\) has each of its four characters occurring once.
Schur's lemma gives the asserted algebras.

Positive rank-one operators correspond to the maps \(2a(\cdot)a^*\), which preserve positivity.
Hence every positive element \(w\) of \((W,\bullet,\ast)\) satisfies
\[
\langle\chi,w(\chi)\rangle
=h_A\bigl(\chi^{1/2}w(\chi)\chi^{1/2}\bigr)\ge0
\qquad(\chi\in A_+).
\]
\end{proof}

For the calculations, order \(X\) by \(0<1<2<3<4\) and use the following basis of \(V_B\), where \(X\setminus B=\{c<d<e\}\):
\[
\left(
\frac{e_c-e_d}{\sqrt2},\
\frac{e_c+e_d-2e_e}{\sqrt6}
\right).
\]
In each component define
\[
\begin{gathered}
\sigma_1:=\begin{pmatrix}0&1\\1&0\end{pmatrix},\qquad
\sigma_2:=\begin{pmatrix}0&-\mathrm i\\\mathrm i&0\end{pmatrix},
\qquad
\sigma_3:=\begin{pmatrix}1&0\\0&-1\end{pmatrix},\\
F_0:=I_2/\sqrt2,\qquad
F_j:=\sigma_j/\sqrt2\quad(1\le j\le3).
\end{gathered}
\]
The \(F_j\) form an orthonormal basis for the ordinary Hilbert--Schmidt inner product.
The operator trace on \(A\) is
\[
\operatorname{Tr}_A(w):=
\sum_{B\in\mathcal B}
\operatorname{Tr}\bigl(
w_{BB}:\End(V_B)\to\End(V_B)
\bigr).
\]

\begin{lemma}\label{lem:johnson-distribution-matrices}
The feasible functionals of the design program~\eqref{def:design-conic-bound} are exactly
\begin{equation}\label{eq:johnson-functional-parameters}
\mathsf A(w)=\operatorname{Tr}_A\!\left(
\left(
xE_0+\frac y5e_{(2,2,1)}+\frac z4e_{(2,1,1,1)}
\right)w\right),
\qquad \mathsf B=Q'\mathsf A,
\end{equation}
where
\begin{equation}\label{eq:johnson-feasible-parameters}
x,y,z\ge0,\qquad x+y+z=1,\qquad
y\le\frac12,\qquad z\le x,\qquad y+z\le2x.
\end{equation}
The objective is \(40x\).
Replacing only \(\mathsf A_a(v)\ge0\) for \(v\in R_a^+\) by \(\mathsf A(z_a)\ge0\), where \(z_a\) are the minimal central projections of \((W,\bullet,\ast)\), replaces the last inequality by \(y\le2x\).
\end{lemma}

\begin{proof}
The design conditions \(\mathsf B_i=0\) for \(i\in T_2\) leave the multiplicity-one modules \(S^{(5)},S^{(2,2,1)},S^{(2,1,1,1)}\), of dimensions \(1,5,4\).
The MacWilliams relation gives \(\mathsf B_i=40\mathsf A|_{U_i}\), so positivity on \(K_i^+\) and normalization give~\eqref{eq:johnson-functional-parameters} with \(x,y,z\ge0\) and \(x+y+z=1\).
Since \(J=40E_0\), the objective is \(40x\).
We compute the three projection maps to determine the remaining constraints.

The map \(E_0\) has coefficient \(1/10\) between any two \(F_0\) coordinates and is zero elsewhere.
The \(F_2\) coordinates carry \(\operatorname{Ind}_{S_2\times S_3}^{S_5}(\mathbf1\boxtimes\operatorname{sgn})\).
Identify them with the exterior basis indexed by the increasing triples \(X\setminus B\).
The projection \(\bigwedge^3(I_5-J_5/5)\) has image \(\bigwedge^3S^{(4,1)}\cong S^{(2,1,1,1)}\)~\cite{Sagan2001}, so
\begin{equation}\label{eq:johnson-exterior-projection}
\begin{aligned}
(e_{(2,1,1,1)})_{BC}(F_2)
&=\det(I_5-J_5/5)_{X\setminus B,\,X\setminus C}\,F_2,\\
(e_{(2,1,1,1)})_{BC}(F_j)
&=0\qquad(j\ne2),
\end{aligned}
\end{equation}
where \(J_5\) is the all-ones matrix.
For \(B=\{0,1\}\) and \(C=\{0,1\},\{0,2\},\{2,3\}\), these determinants are \(2/5,-1/5,0\).

For \(e_{(2,2,1)}\), use the orbit of \(P^+\).
Extend \(P_B^+\) by zero on \(V_B^\perp\subseteq\C^{X\setminus B}\).
Its unique diagonal entry \(2/3\) occurs at \(m_B\), so any permutation stabilizing \(P^+\) preserves all midpoints.
Hence
\[
g(a+b)+g(0)
=2g((a+b)/2)
=g(a)+g(b),
\]
so \(x\mapsto g(x)-g(0)\) is additive on \(\mathbb F_5\).
Thus \(g(x)=g(0)+(g(1)-g(0))x\), and \(g\) is affine.
The stabilizer is therefore \(K\), and the orbit has six elements.
Define
\[
\{H_1,\ldots,H_6\}
:=\{g\cdot P^+-\tfrac12 1_A:g\in S_5\}.
\]
These elements are real symmetric and componentwise traceless, so they lie in \(\operatorname{Ind}_{S_2\times S_3}^{S_5}(\mathbf1\boxtimes\operatorname{Std})\).
They satisfy \(e_\lambda H_r=0\) for \(\lambda\in T_2\), by Theorem~\ref{thm:johnson-minimum-size}; their nonzero invariant span is therefore the single copy of \(S^{(2,2,1)}\).
For the ordinary Hilbert--Schmidt inner product, the equivariant sum of their rank-one operators is scalar on this five-dimensional module by Schur's lemma.
Since \(\sum_B\operatorname{Tr}(H_{r,B}^*H_{r,B})=5\), its trace is \(6\cdot5\), giving
\begin{equation}\label{eq:johnson-five-projection}
e_{(2,2,1)}(\chi)
=\frac16\sum_{r=1}^6H_r
\sum_{B\in\mathcal B}\operatorname{Tr}(H_{r,B}^*\chi_B)
\qquad(\chi\in A).
\end{equation}
The factor \(2\) in \(h_A\) does not change orthogonal projections.

The following coordinates make this calculation explicit:
\[
2H_{r,B}
=n_{r,B}^{\mathsf T}\binom{\sigma_1}{\sigma_3},
\qquad n_{r,B}\in\R^2.
\]
For complement \(\{c<d<e\}\), midpoint choices \(c,d,e\) give
\[
n_{r,B}
=\binom{\sqrt3/2}{1/2},\quad
\binom{-\sqrt3/2}{1/2},\quad
\binom0{-1},
\]
respectively.
The six permutations fixing \(0,1\) individually give distinct orbit elements, with midpoint choices
\[
\begin{array}{c|ccc}
r&\{0,1\}&\{0,2\}&\{2,3\}\\ \hline
1&2&3&1\\
2&2&4&0\\
3&3&1&0\\
4&3&4&1\\
5&4&1&4\\
6&4&3&4
\end{array}
\]
because only the identity in \(K\) fixes both points.
For the three pair representatives, respectively,
\[
\sum_rn_{r,B}n_{r,C}^{\mathsf T}
=3I_2,\quad
-\tfrac32I_2,\quad
\begin{pmatrix}0&0\\0&3\end{pmatrix}.
\]
Equation~\eqref{eq:johnson-five-projection} divides these by \(12\): by \(6\) for the orbit sum and by \(2\) for the normalized basis \((F_1,F_3)\).

To express \(\mathsf A\) on the ideals for \(\bullet\), define coefficients \(d_{\mu\nu}\) and operators \(D_{BC}\in\End(\operatorname{Hom}(V_C,V_B))\) by
\[
\begin{aligned}
\left(
xE_0+\frac y5e_{(2,2,1)}+\frac z4e_{(2,1,1,1)}
\right)_{CB}(F_\nu)
&=\sum_{\mu=0}^3 d_{\mu\nu}F_\mu,\\
D_{BC}(a)
&:=2\sum_{\mu,\nu=0}^3d_{\mu\nu}F_\nu aF_\mu.
\end{aligned}
\]
Here \(0\le\mu,\nu\le3\), \(F_\nu\in\End(V_B)\), and \(F_\mu\in\End(V_C)\).
Tracing the composition of this \((C,B)\)-component with \(2a(\cdot)b^*\) gives \(\operatorname{Tr}(b^*D_{BC}(a))\), hence
\begin{equation}\label{eq:johnson-component-pairing}
\mathsf A(w)
=\sum_{B,C}\operatorname{Tr}\bigl(D_{BC}[w_{BC}]\bigr).
\end{equation}
Here the trace is on \(\operatorname{Hom}(V_C,V_B)\).
At the three representatives above,
\begin{equation}\label{eq:johnson-component-operators}
\begin{aligned}
D_{BB}(a)
&=\frac{x}{10}a
 +\frac y{20}(\sigma_1a\sigma_1+\sigma_3a\sigma_3)
 +\frac z{10}\sigma_2a\sigma_2,\\
D_{BC}(a)
&=\frac{x}{10}a
 -\frac y{40}(\sigma_1a\sigma_1+\sigma_3a\sigma_3)
 -\frac z{20}\sigma_2a\sigma_2
&&(|B\cap C|=1),\\
D_{BC}(a)
&=\frac{x}{10}a+\frac y{20}\sigma_3a\sigma_3
&&(B\cap C=\varnothing).
\end{aligned}
\end{equation}
For \(B=\{0,1\}\) and \(C=\{0,2\}\), choose a real orthonormal eigenbasis of \(V_B\) for the nonidentity element of the pair stabilizer, with eigenvalues \(1,-1\), and transport it to \(V_C\) by \((1\,2)\).
This transposition matches the initial bases, so the same real orthogonal change applies in both components.
It preserves \(\operatorname{span}_{\R}\{\sigma_1,\sigma_3\}\) orthogonally and sends \(\sigma_2\) to \(\pm\sigma_2\), leaving the middle expression in~\eqref{eq:johnson-component-operators} unchanged.
In these bases, identify \(\operatorname{Hom}(V_C,V_B)\) with \(M_2(\C)\), and let \(E_{ij}\) denote its matrix units.
The nonidentity element of the pair stabilizer has eigenspaces with ordered bases \((E_{11},E_{22})\) and \((E_{12},E_{21})\).
In these bases, the restrictions of \(D_{BC}\) have matrices
\begin{equation}\label{eq:johnson-two-matrices}
\frac1{40}
\begin{pmatrix}
4x-y&-y-2z\\
-y-2z&4x-y
\end{pmatrix},
\qquad
\frac1{40}
\begin{pmatrix}
4x+y&-y+2z\\
-y+2z&4x+y
\end{pmatrix}.
\end{equation}
Diagonalizing gives
\begin{equation}\label{eq:johnson-density-eigenvalues}
\begin{array}{c|c|c}
|B\cap C|&\text{eigenvalues}&\text{multiplicities}\\ \hline
2&
\dfrac1{10},\
\dfrac{x-y+z}{10},\
\dfrac{x-z}{10}
&1,1,2\\[4pt]
1&
\dfrac{2x-y-z}{20},\
\dfrac{2x+z}{20},\
\dfrac{2x+y-z}{20}
&1,2,1\\[4pt]
0&
\dfrac{2x+y}{20},\
\dfrac{2x-y}{20}
&2,2
\end{array}
\end{equation}

For the functionals in~\eqref{eq:johnson-functional-parameters}, we now verify that these eigenvalue tests are equivalent to positivity on \(R_a^+\).
For \((V_a,\bullet,\ast)\cong\C\), Proposition~\ref{prop:johnson-schur-decomposition} gives \(z_a\in R_a^+\).
The nonzero positive map \(z_a\) cannot annihilate \(1_A\), so
\(\langle1_A,z_a(1_A)\rangle>0\) and
\[
R_a^+=\R_{\ge0}z_a.
\]
For either simple ideal of \((W,\bullet,\ast)\) isomorphic to \(M_2(\C)\), the unit eigenvectors of \(D_{BC}\) in~\eqref{eq:johnson-two-matrices} are
\[
q=\frac{E_{11}\pm E_{22}}{\sqrt2}
\quad\text{or}\quad
q=\frac{E_{12}\pm E_{21}}{\sqrt2}.
\]
Take \(u=(1,1)^{\mathsf T}\in V_B\) and \(v=(1,\pm1)^{\mathsf T}\in V_C\) in the chosen bases, with the sign such that the projection of \(uv^*\) onto the relevant eigenspace is \(\pm\sqrt2q\), and set
\[
\chi_B:=uu^*,\qquad
\chi_C:=vv^*,\qquad
\chi_D:=0\quad(D\notin\{B,C\}).
\]
Equivariance under the transposition exchanging \(B,C\) identifies \([w_{BC}]\) with \([w_{CB}]\) in these bases.
The two ordered-pair contributions, with the trace and correspondence normalizations, give
\[
0\le\langle\chi,w(\chi)\rangle
=16\langle q,[w_{BC}]q\rangle
\qquad(w\in R_a^+).
\]
Thus each functional \(w\mapsto\langle q,[w_{BC}]q\rangle\) is nonnegative on \(R_a^+\).
Equivariance and~\eqref{eq:johnson-component-pairing} then show that nonnegative eigenvalues suffice for \(\mathsf A_a\ge0\) on \(R_a^+\).
Conversely, the rank-one projection onto each displayed \(q\) commutes with the stabilizer and extends equivariantly to an element positive for \(\bullet\), hence in \(R_a^+\).
Testing it forces the corresponding eigenvalue to be nonnegative.
This proves the equivalence; the Hermitian operators in~\eqref{eq:johnson-functional-parameters} and~\eqref{eq:johnson-component-operators} give the required involution identities.

Using \(x+y+z=1\), nonnegativity in~\eqref{eq:johnson-density-eigenvalues} is equivalent to~\eqref{eq:johnson-feasible-parameters}.
Testing only the minimal central projections of \((W,\bullet,\ast)\) replaces positive semidefiniteness of the two matrices in~\eqref{eq:johnson-two-matrices} by their trace conditions \(4x-y\ge0\) and \(4x+y\ge0\).
The retained inequality \(2x-y\ge0\) from \(B\cap C=\varnothing\) implies both, leaving exactly the weaker inequality \(y\le2x\).
\end{proof}

\begin{theorem}\label{thm:johnson-conic-values}
The design program satisfies
\[
\DelsDes(W,T_2)=\frac{40}{3}.
\]
Replacing only the inequalities \(\mathsf A_a(v)\ge0\) for \(v\in R_a^+\) by \(\mathsf A(z_a)\ge0\), where \(z_a\) are the minimal central projections of \((W,\bullet,\ast)\), gives \(10\).
\end{theorem}

\begin{proof}
For the full program, Lemma~\ref{lem:johnson-distribution-matrices} gives
\[
1=x+y+z\le3x,\qquad
40x\ge\frac{40}{3},
\]
with equality at the feasible point \((x,y,z)=(1/3,1/2,1/6)\).
For the weaker program,
\[
1=x+y+z\le2x+\frac12,\qquad
40x\ge10,
\]
with equality at the feasible point \((x,y,z)=(1/4,1/2,1/4)\).
\end{proof}

Thus
\[
10<\frac{40}{3}<20.
\]
Positive semidefiniteness of the two matrices in~\eqref{eq:johnson-two-matrices} gives the first strict inequality; the second shows that the conic constraints do not capture the minimum effective size of positive elements satisfying the design conditions.
\section*{Acknowledgements}

The author is grateful to Professors Hajime Tanaka and Takayuki Okuda for their valuable comments and helpful discussions.
This work was supported by JSPS KAKENHI Grant Number JP26KJ0570.

\providecommand{\bysame}{\leavevmode\hbox to3em{\hrulefill}\thinspace}
\providecommand{\MR}{\relax\ifhmode\unskip\space\fi MR }
\providecommand{\MRhref}[2]{%
  \href{http://www.ams.org/mathscinet-getitem?mr=#1}{#2}
}
\providecommand{\href}[2]{#2}

\end{document}